\documentclass[a4paper,12pt]{article}

\usepackage{enumitem}
\usepackage[english]{babel}
\usepackage[latin1]{inputenc}
\usepackage{amsmath,amssymb,amsfonts,mathrsfs}
\usepackage[amsmath,thmmarks]{ntheorem}
\usepackage{flexisym}
\usepackage{breqn}

\usepackage{graphicx}
\usepackage{tikz}
\usepackage[export]{adjustbox}
\usepackage{wrapfig}
\usetikzlibrary{shapes.geometric,arrows}
\usetikzlibrary{matrix,chains,positioning,decorations.pathreplacing,arrows}
\usepackage{titling}
\usepackage{url}
\usepackage{geometry}
\usepackage{nicefrac}

\usepackage[nocompress]{cite}
\usepackage{hyperref}

\usepackage{mathtools}

\usepackage{float}
\usepackage{caption}

\usepackage[h]{esvect}

\usepackage{array}

\usepackage{listings}
\numberwithin{equation}{section}
\newtheorem{theorem}{Theorem}[section]

\newtheorem{corollary}[theorem]{Corollary}
\newtheorem{definition}[theorem]{Definition}
\newtheorem{lemma}[theorem]{Lemma}
\newtheorem{proposition}[theorem]{Proposition}
\newtheorem{setting}[theorem]{Setting}
\newtheorem{notation}[theorem]{Notation}

\makeatletter
\renewtheoremstyle{plain}% Adds automatic line break, if heading is too long
{\item{\theorem@headerfont ##1\ ##2\theorem@separator}~}
{\item{\theorem@headerfont ##1\ ##2\ (##3)\theorem@separator}~}
\makeatother

\theoremstyle{nonumberplain}
\theorembodyfont{\normalfont}
\theoremsymbol{\ensuremath{\square}}
\newtheorem{proof}{\textit{Proof.}}

\newcommand{\N}{\mathbb{N}}

\newcommand{\R}{\mathbb{R}}

\newcommand{\E}{\mathbb{E}}
\renewcommand{\epsilon}{\ensuremath\varepsilon}
\newcommand{\Id}{\operatorname{I}}

\newcommand{\ResReal}{\mathscr{R}}
\newcommand{\ResNet}{\mathscr{N}}
\newcommand{\ResPar}{\mathscr{P}}
\newcommand{\ResLayer}{\mathscr{L}}
\newcommand{\ResDims}{\mathscr{D}}
\newcommand{\ResParal}{\mathbb{P}}

\DeclarePairedDelimiter\abs{\lvert}{\rvert}
\DeclarePairedDelimiter\norm{\lVert}{\rVert}

\makeatletter

\newcommand{\tightoverset}[2]{%
	\mathop{#2}\limits^{\vbox to -.5ex{\kern-0.75ex\hbox{$#1$}\vss}}}

\usepackage{titlesec}

\usepackage[nosort,capitalise]{cleveref}

\usepackage{xcolor}
\hypersetup{
	colorlinks,
	linkcolor={red!60!black},
	citecolor={green!60!black},
	urlcolor={blue!60!black}
}

\begin{document}

\title{Approximation properties of Residual\\ Neural
	 Networks for Kolmogorov PDEs}

\author{Jonas Baggenstos$^{1}$ and Diyora Salimova$^{2}$
	\bigskip
	\\
	\small{$^1$Department of Mathematics, ETH Zurich,}\\
	\small{Switzerland, e-mail:  jonas.baggenstos@gmx.ch}
\\
\small{$^2$Department of Information Technology and Electrical Engineering,}\\
\small{ETH Zurich, Switzerland, e-mail:  sdiyora@ethz.ch}}

\maketitle

\begin{abstract}
  In recent years residual neural  networks (ResNets) as introduced by He et al. \cite{he2016deep} have become very popular in a large number  of applications, including in image classification and segmentation. They provide a new perspective in training very deep neural networks without suffering  the vanishing gradient problem.  In this article we show that ResNets are able to approximate solutions of Kolmogorov partial differential equations (PDEs) with constant diffusion  and possibly nonlinear drift coefficients without  suffering the curse of dimensionality, which is to say the number of parameters of the approximating ResNets grows at most polynomially in the reciprocal of the approximation accuracy $\epsilon > 0$ and the dimension of the considered PDE $d\in\N$. We adapt a proof in Jentzen et al.\ \cite{jentzen2018proof} - who showed a similar result for  feedforward neural networks (FNNs) - to ResNets. In contrast to FNNs, the Euler-Maruyama approximation structure of ResNets simplifies the construction of the approximating ResNets substantially. Moreover, contrary to \cite{jentzen2018proof}, in our proof using ResNets does not require the existence of an FNN (or a ResNet) representing the identity map, which  enlarges the set of applicable activation functions. 
\end{abstract}

\tableofcontents

\section*{Notation}

\begin{notation}[Natural and real numbers]
	We denote the set of natural numbers by $\N = \{1,\allowbreak 2,\allowbreak 3,\allowbreak \dots \}$, the set of natural numbers including zero by $\N_0 =\{0,1,2,3,\dots\}$, and the set of real numbers by $\R$.
	\end{notation}

\begin{notation}[Scalar product and Euclidean norm]
	Let $n\in\N$ and $x=(x_1,x_2,\dots,x_n),\allowbreak y=(y_1,y_2,\dots,y_n)\in\R^n$. Then we denote by $\langle\cdot,\cdot\rangle \colon \mathbb{R}^{n} \times \R^n\to\R$ the standard scalar product given by
	\begin{align*}
		\langle x,y\rangle = \textstyle\sum\limits_{i=1}^n x_iy_i.
	\end{align*}
	Moreover, we denote by $\norm{x}$ the Euclidean norm of $x\in\R^n$ induced by the standard scalar product, i.e.,
	\begin{align*}
		\norm{x} = \bigg[\textstyle\sum\limits_{i=1}^n x_i^2 \bigg]^{\nicefrac{1}{2}}.
	\end{align*}
\end{notation}
\begin{notation}[Identity matrix]\label{def5.8}
	Let $n \in \mathbb{N}$. Then we denote by $\Id_n \in \mathbb{R}^{n \times n}$ the identity matrix in $\mathbb{R}^{n \times n}$. %and \textcolor{red}{by $O_n\in\R^{n\times n}$ the zero matrix in $\R^{n\times n}$.}
\end{notation}
\begin{notation}[Multidimensional versions]\label{def5.2}
	Let $f \colon \mathbb{R} \rightarrow \mathbb{R}$ be a function. Then we denote by $\mathcal{M}_{f} \colon (\cup_{n \in \mathbb{N}} \mathbb{R}^n) \rightarrow (\cup_{n \in \mathbb{N}} \mathbb{R}^n)$ the function which satisfies for all $n \in \mathbb{N}$, $x=(x_1,x_2,...,x_n)\in \mathbb{R}^n$ that
	\begin{align*}
	\mathcal{M}_{f}(x) = (f(x_1),f(x_2),...,f(x_n)).
	\end{align*}
\end{notation}

\begin{notation}[Gradient and Hessian matrix]
	Let $n\in\N$ and $f\colon\R^n\to\R$ be a  differentiable function. Then we denote the partial derivatives (also called partials) in the $i$-th variable by
	\begin{align*}
		\partial_{x_i} f, \qquad i \in \{1, 2, \ldots, n\}.
	\end{align*}
	Furthermore, we denote the gradient and, if $f$ is twice differentiable, the Hessian matrix by
	\begin{align*}
		\nabla f = \begin{pmatrix}
		\partial_{x_1} f\\
		\partial_{x_2} f\\
		\vdots\\
		\partial_{x_n} f
		\end{pmatrix} \qquad \text{and} \qquad
		\operatorname{Hess}(f) = 
		\begin{pmatrix}
			\partial_{x_1}^2f & \partial_{x_1}\partial_{x_2}f & \dots & \partial_{x_1}\partial_{x_n}f\\
			\partial_{x_2}\partial_{x_1}f & \partial_{x_2}^2 f& \dots& \partial_{x_2}\partial_{x_n} f\\
			\vdots&\vdots&\ddots&\vdots\\
			\partial_{x_n}\partial_{x_1} f& \partial_{x_n}\partial_{x_2} f & \dots & \partial_{x_n}^2 f
		\end{pmatrix}.
	\end{align*}
	\end{notation}

\begin{notation}[Set of continuous functions]
Let $m, n \in \N$. Then we denote by $C(\R^m, \R^n)$ the set of all continuous functions $f \colon \R^m \to \R^n$.
\end{notation}

\addcontentsline{toc}{section}{Notation}
\pagebreak

\section{Introduction}

 In recent years residual neural  networks (ResNets) as introduced by He et al.\ \cite{he2016deep} have become very popular in a large number of applications, especially, in image classification and segmentation. One of their main advantages is that even very deep ResNets are able to be trained without suffering the vanishing gradient problem, see, e.g., \cite{10.1007/978-3-319-66471-2_8,he2016deep,he2016identity,orhan2018skip,wu2016wider}. In this article, in \Cref{maincor} below, we show that ResNets have the capacity to approximate solutions of Kolmogorov partial differential equations (PDEs) (see, e.g., \cite{kolmogoroff1931analytischen,krylov2006stochastic}) with constant diffusion  and possibly nonlinear drift coefficients without suffering the curse of dimensionality (CoD), which is to say the number of parameters of the approximating ResNets grows at most polynomially in the reciprocal of the approximation accuracy $\epsilon > 0$ and the dimension of the considered PDE $d\in\N$. We achieve this by  adapting a proof in Jentzen et al.\ \cite{jentzen2018proof} - who showed a similar result for  feedforward neural networks (FNNs) - to ResNets.
  
  We define ResNets as a sequence of FNNs and linear projections of the skip connections, also called shortcuts (cf. \Cref{DefResNet,ReaResNet} below). In this way ResNets can be seen as Euler-Maruyama approximations if we set the linear projection equal to the identity \cite{avelin2021neural,muller2020space}. Due to the similar structure to the Euler-Maruyama approximation, ResNets perfectly suit the proof structure presented in \cite{jentzen2018proof}. Therefore, we are able to show that ResNets approximating Kolmogorov PDEs with constant diffusion  and possibly nonlinear drift coefficient overcome the CoD. We  note that there are many more theoretical results about neural networks approximating solutions of PDEs in the scientific literature, see, e.g., \cite{beck2018solving,beck2019machine,grohs2020deep,han2018solving,jentzen2018proof,weinan2018deep,hornung2020spacetime,GrohsJentzenSalimova2019}.

Our proof of the main result of this article, \Cref{mainprop} below, involves several steps. With the help of the Feynman-Kac formula we obtain suitable stochastic processes linked to   solutions of  Kolmogorov PDEs. The weak error estimates for an Euler-Maruyama type approximation as presented in \cite[Proposition 4.2]{jentzen2018proof} and  Monte Carlo type estimates as in \cite[Corollary 2.5]{grohs2018proof} then yield to  random fields with the desired approximation accuracy. Lemma 2.1 in \cite{jentzen2018proof} assures the existence of  realizations of these random fields (which are deterministic) with the same approximation accuracy. Assuming that we can approximate the drift and the initial condition function of the PDE by FNNs without suffering the CoD, we then construct  ResNets with realizations equal to the obtained approximating functions. As the last step we show the polynomial complexity bounds of the constructed ResNets.

The Euler-Maruyama type structure induced by the definition of ResNets simplifies the construction of the ResNet in our proof substantially compared to \cite{jentzen2018proof}. On one hand, a construction as in \cite[Proposition 5.2]{jentzen2018proof} to perform one step in the Euler-Maruyama approximation becomes unnecessary. On the other hand, due to our complexity measure on ResNets, plugging in the artificial identity between two neural networks is not needed anymore to obtain polynomial complexity bounds. Therefore, we do not need to assume that FNNs have the capability to describe the identity function, which, would be an additional assumption on the activation function. This, of course,  enlarges the set of applicable activation functions.

The article is structured as follows. In Section \ref{FFDNNs}, we recall definitions and results on  FNNs, closely following \cite{jentzen2018proof}. Using results from Section \ref{FFDNNs} and inspired by \cite{he2016identity,muller2020space},  in Section \ref{ResidualNetworks} we develop a theory involving different manipulations of ResNets and FNNs. In Section \ref{ResNetPDE}, we state and prove our main result, \Cref{mainprop} and its slight extension - \Cref{maincor}. Finally, in Section \ref{conc} we provide some closing comments.

\section{Feedforward neural networks (FNNs)}\label{FFDNNs}

In this section, we recall some  definitions and results on standard FNNs from the scientific literature (cf., e.g., \cite[Section 5]{jentzen2018proof}). These results on FNNs allow us to obtain similar results on ResNets presented in Section \ref{ResidualNetworks} below. For proofs of the results of this section we refer to, e.g., Grohs et al.~\cite[Section 2]{grohs2019spacetime} and Jentzen et al.~\cite[Section 5]{jentzen2018proof}.

\begin{definition}[FNNs]\label{def5.1}
	We denote by $\mathcal{N}$ the set given by
	\begin{align}
	\mathcal{N} = \cup_{L \in \N} \cup_{l_0,l_1,\dots,l_L \in \N} ( \times_{k=1}^{L}( \R^{l_k \times l_{k-1}} \times \R^{l_k}))
	\end{align}
	and we denote by $\mathcal{P}, \mathcal{L},\mathcal{I},\mathcal{O} \colon \mathcal{N} \rightarrow \N,$ $\mathcal{D} \colon \mathcal{N} \rightarrow (\cup_{L \in \N} \N^L)$, and $\mathbb{D}_n  \colon \mathcal{N} \rightarrow \N_0$, $n \in \N_0$, the functions which satisfy for all $L \in \N$, $l_0, l_1, . . . , l_L \in \N$, $\theta = ((W_1,B_1),(W_2,B_2),\dots,(W_L,B_L)) \in (\times_{k=1}^{L}( \R^{l_k \times l_{k-1}} \times \R^{l_k}))$, $n \in \N_0$ that $\mathcal{L}(\theta) =L$, $\mathcal{I}(\theta) =l_0$, $\mathcal{O}(\theta) =l_L$, and
	\begin{align*}
	\mathcal{P}(\theta) &= \textstyle\sum\limits_{k=1}^{L} l_k (l_{k-1} + 1),\\
	\mathcal{D}(\theta) &= (l_0,l_1,\dots,l_L),\\
	\mathbb{D}_n(\theta) &= \begin{cases}l_n \quad & \colon n \leq L \\ 0 & \colon n > L. \end{cases}
	\end{align*}
\end{definition}
We call $\theta\in\mathcal{N}$ an FNN and we call the number of layers $\mathcal{L}(\theta) = L$ the depth of the FNN $\theta$, and denote by $l_i$ the size of layer $i$ and this represents the number of neurons in layer $i$. Furthermore, we call $\mathcal{I}(\theta) = l_0$ the input dimension, $\mathcal{O}(\theta) = l_L$ the output dimension, and we call $\mathcal{P}(\theta)$ the complexity of the FNN $\theta \in\mathcal{N}$. One usually refers to $\mathcal{D}(\theta)$, respectively to $\{\mathbb{D}_n(\theta)\}_{n \in \N_0}$, as the architecture of the FNN $\theta \in\mathcal{N}$.

Note that there are different definitions for the complexity of an FNN, e.g., the depth, the size of the largest layer, the total number of neurons or the number of non-zero parameters, i.e., non-zero entries of all weights and biases, see, e.g., \cite{elbrachter2020dnn,gribonval2021approximation,grohs2019spacetime,jentzen2018proof,muller2020space,petersen2018optimal}. However, the number of non-zero parameters and the total number of parameters are related in the sense that a polynomial bound on the complexity of a non-degenerate FNN considering all parameters leads to a polynomial bound on the complexity of the non-degenerate FNN considering only non-zero parameters and vice versa (see, e.g., Elbr\"achter et al.~\cite[Section 5]{elbrachter2020dnn} for more details).

\begin{definition}[Realizations associated to FNNs]\label{def5.3}
	Let $a \in C(\R,\R)$. Then we denote by $\mathcal{R}_a  \colon \mathcal{N} \rightarrow (\cup_{u,v \in \N} C(\R^u, \R^v))$ the function which satisfies for all $L \in \N$, $l_0, l_1, \dots , l_L \in \N,$ $\theta = ((W_1,B_1),(W_2,B_2),\dots,(W_L,B_L)) \in (\times_{k=1}^{L}( \R^{l_k \times l_{k-1}} \times \R^{l_k}))$, $x_0 \in \R^{l_0}$, $x_1 \in \R^{l_1},\dots, x_{L-1} \in \R^{l_{L-1}}$ with for all $k \in \N \cap (0,L)$
	\begin{align*}
		x_k=\mathcal{M}_a(W_k x_{k-1} +B_k)
	\end{align*} 
	that $\mathcal{R}_a \theta \in C(\R^{l_0}, \R^{l_L})$ and $(\mathcal{R}_a\theta)(x_0) = W_L x_{L-1} +B_L$ (cf.\ Notation \ref{def5.2} and Definition \ref{def5.1}).
\end{definition}

\begin{figure}[H]
	\centering
	\begin{tikzpicture}
	\tikzstyle{input} = [rectangle, minimum width=0.001cm, minimum height=0.8cm,text centered, draw=white]
	\tikzstyle{output} = [rectangle, minimum width=0.001cm, minimum height=0.8cm,text centered, draw=white]
	\tikzstyle{weights} = [rectangle, minimum width=2cm, minimum height=0.4cm,text centered, draw=black]
	\tikzstyle{bias} = [rectangle, minimum width=0.8cm, minimum height=0.4cm, text centered, draw=black]
	\tikzstyle{a} = [rectangle, minimum width=0.8cm, minimum height=0.4cm, text centered, draw=black]
	\tikzstyle{arrow} = [thick,->,>=stealth]
	
	\node (input) [input] {};
	\node (weight1) [weights, right of=input, xshift=2cm] {$W_1 x_0 + B_1$};
	\node(a1)[a,right of=weight1,xshift=1cm]{$\mathcal{M}_a$};
	\node (weight2)[weights,right of=a1,xshift=1cm]{$W_2 x_1 +B_2$};
	\node (a2)[a,right of=weight2,xshift=1cm]{$\mathcal{M}_a$};
	\node (input2) [input,right of=a2,xshift=0.2cm] {$\cdots$};
	\node (weightL) [weights, right of=input2, xshift=1cm] {$W_Lx_{L-1} + B_L$};
	\node (output) [output,right of=weightL,xshift=2cm] {};
	
	\draw [arrow] (input) --node[anchor = south]{Input $x_0$}(weight1);
	\draw [arrow] (weight1) -- (a1);
	\draw [arrow] (a1) -- (weight2);
	\draw [arrow] (weight2) -- (a2);
	\draw [arrow] (a2) --(input2);
	\draw [arrow] (input2) -- (weightL);
	\draw [arrow] (weightL) --node[anchor = south]{Output}(output);
	\end{tikzpicture}
	\caption{Realization of an FNN.}
	\label{figrealizationDNN}
\end{figure}
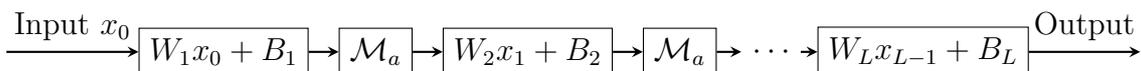

Intuitively speaking, $\mathcal{R}_a \theta$ is the function the FNN $\theta \in \mathcal{N}$ and the function $a \in C(\R,\R)$ represent. Here, one usually calls $a \in C(\R,\R)$ the activation function. The first layer receives an input signal $x_0$, performs a weighting with weights $W_1$ and adds the bias $B_1$. Then the activation function $a$ is applied in every component and $x_1$ is sent as input to the second layer. This procedure is repeated until the last layer is reached, where the signal $W_L x_{L-1} +B_L$ is sent as an output directly without applying the activation function (see Figure \ref{figrealizationDNN}). There are many commonly used activation functions, see, e.g., \cite{doi:10.1142/S0129065709002130,he2015delving,panigrahi2015navigation,ramachandran2017searching,wang2020influence} for further discussion.

\begin{definition}[Standard compositions of FNNs]\label{def5.5}
	We denote by $(\cdot)\bullet (\cdot) \colon \{(\theta_1,\theta_2)\in \mathcal{N} \times \mathcal{N}\mid \mathcal{I}(\theta_2) = \mathcal{O}(\theta_1)\} \rightarrow \mathcal{N}$ the function which satisfies for all $\mathfrak{L},L \in \N$, $\mathfrak{l}_0, \mathfrak{l}_1, \dots,\mathfrak{l}_{\mathfrak{L}}$, $l_0, l_1, \dots, l_L \in \N$,
	\begin{align*}
	\theta_1 &= ((\mathcal{W}_1,\mathfrak{B}_1),(\mathcal{W}_2,\mathfrak{B}_2), \dots, (\mathcal{W}_{\mathfrak{L}},\mathfrak{B}_{\mathfrak{L}}))\in (\times_{k=1}^{\mathfrak{L}}( \R^{\mathfrak{l}_k \times \mathfrak{l}_{k-1}}\times \R^{\mathfrak{l}_k})),\\
	\theta_2 &= ((W_1,B_1),(W_2,B_2),\dots,(W_L,B_L)) \in (\times_{k=1}^{L}( \R^{l_k \times l_{k-1}} \times \R^{l_k}))
	\end{align*}
	 with $l_0 =\mathcal{I}(\theta_2) = \mathcal{O}(\theta_1) = \mathfrak{l}_{\mathfrak{L}}$ that
	\begin{align*}
	\theta_2 \bullet \theta_1 =& \left((\mathcal{W}_1,\mathfrak{B}_1),(\mathcal{W}_2,\mathfrak{B}_2),\dots,(\mathcal{W}_{\mathfrak{L}-1},\mathfrak{B}_{\mathfrak{L}-1}),\right.\\ &\left.(W_1 \mathcal{W}_{\mathfrak{L}},W_1 \mathfrak{B}_{\mathfrak{L}} +B_1),(W_2,B_2),(W_3,B_3),\dots,(W_L,B_L)\right)
	\end{align*}
	(cf.\ Definition \ref{def5.1}).
\end{definition}

\begin{lemma}\label{lem4.6}
	Let $\theta_1,\theta_2 \in \mathcal{N}$ be as in Definition \ref{def5.5} (cf.\ Definition \ref{def5.1}). Then
	\begin{enumerate}[label = \roman*)]
		\item\label{lem4.61} it holds that
		\begin{align}
		\mathcal{D}(\theta_2 \bullet \theta_1) = (\mathbb{D}_0(\theta_1),\mathbb{D}_1(\theta_1),\dots,\mathbb{D}_{\mathcal{L}(\theta_1)-1}(\theta_1),\mathbb{D}_1(\theta_2),\mathbb{D}_2(\theta_2),\dots,\mathbb{D}_{\mathcal{L}(\theta_2)}(\theta_2)),
		\end{align}
		and
		\item\label{lem4.62} it holds for all $a \in C(\R,\R)$ that $\mathcal{R}_a (\theta_2 \bullet \theta_1) \in C(\R^{\mathcal{I}(\theta_1)}, \R^{\mathcal{O}(\theta_2)})$
		and
		\begin{align}
		\mathcal{R}_a (\theta_2 \bullet \theta_1) = [\mathcal{R}_a (\theta_2)] \circ [\mathcal{R}_a (\theta_1)]
		\end{align}
	\end{enumerate}
	(cf.\ Definitions \ref{def5.3} and \ref{def5.5}).
\end{lemma}

	Statement \ref{lem4.62} in Lemma \ref{lem4.6} above shows the natural property that the realization of two composed FNNs is the composition of the realization of the two FNNs. On the other side, \ref{lem4.61} shows that $\mathfrak{l}_{\mathcal{L}(\theta_1)-1}l_1 \leqslant \mathcal{P}(\theta_2\bullet\theta_1)$ and thus, in \cite{jentzen2018proof}, $\mathcal{P}(\theta_2\bullet\theta_1)$ potentially might  grow  as $\mathcal{P}(\theta_2)\mathcal{P}(\theta_1)$ leading to an exponential complexity. This is solved in \cite{jentzen2018proof} by plugging in the artificial identity between the two FNNs (cf.\ \cite[Proposition 5.1, Proposition 5.2]{jentzen2018proof}). %\textcolor{red}{In this thesis we use a definition of Residual Networks which allows to compose ANNs, later called Residual blocks with the complexity as the sum of the composed ANNs and the parameters of the shortcut connections what leads to a polynomial bound in Proposition \ref{mainprop} (cf.\ Definition \ref{DefResNet}, Remark \ref{Remark}).}

\begin{definition}[Parallelizations of FNNs with the same depth]\label{def5.7}
	Let $n \in \N$. Then we denote by
	\begin{align}
	\mathbf{P}_n  \colon \left\{(\theta_1,\theta_2,\dots,\theta_n)\in \mathcal{N}^n  \middle\vert \mathcal{L}(\theta_1)=\mathcal{L}(\theta_2)=\dots=\mathcal{L}(\theta_n) \right\} \rightarrow \mathcal{N}
	\end{align}
	the function which satisfies for all $L \in \N$, $l^1_{0}, l^1_{1}, . . . , l^1_{L}$, $l^2_{0},
	l^2_{1},\dots , l^2_{L}$, \dots, $l^n_{0}, l^n_{1},\dots , l^n_{L} \in \N$, 
	\begin{align*}
	\theta_1 &= ((W^1_{1},B^1_{1}),(W^1_{2},B^1_{2}),\dots,(W^1_{L},B^1_{L})) \in (\times_{k=1}^{L}( \R^{l^1_{k} \times l^1_{k-1}}\times \R^{l^1_{k}})),\\ 
	\theta_2 &= ((W^2_{1},B^2_{1}),(W^2_{2},B^2_{2}),\dots,(W^2_{L},B^2_{L})) \in (\times_{k=1}^{L}( \R^{l^2_{k} \times l^2_{k-1}}\times \R^{l^2_{k}})),\\
	 &\ \vdots\\
	  \theta_n &= ((W^n_{1},B^n_{1}),(W^n_{2},B^n_{2}),\dots,(W^n_{L},B^n_{L})) \in (\times_{k=1}^{L}( \R^{l^n_{k} \times l^n_{k-1}}\times \R^{l^n_{k}}))
	\end{align*}
	 that
	\begin{align}
	\mathbf{P}_n(\theta_1,\theta_2,\dots,\theta_n) = \bBigg@{5}(& \left. \left(
	\begin{pmatrix}
	W^1_{1} & 0 & 0 & \cdots & 0 \\
	0 & W^2_{1} & 0 & \cdots & 0 \\
	0 & 0 & W^3_{1} & \cdots & 0 \\
	\vdots & \vdots & \vdots & \ddots & \vdots \\
	0 & 0 & 0 & \cdots & W^n_{1}
	\end{pmatrix},
	\begin{pmatrix}
	B^1_{1} \\
	B^2_{1} \\
	B^3_{1} \\
	\vdots \\
	B^n_{1}
	\end{pmatrix} \right) ,\right.\\ &\left.\left(
	\begin{pmatrix}
	W^1_{2} & 0 & 0 & \cdots & 0 \\
	0 & W^2_{2} & 0 & \cdots & 0 \\
	0 & 0 & W^3_{2} & \cdots & 0 \\
	\vdots & \vdots & \vdots & \ddots & \vdots \\
	0 & 0 & 0 & \cdots & W^n_{2}
	\end{pmatrix},
	\begin{pmatrix}
	B^1_{2} \\
	B^2_{2} \\
	B^3_{2} \\
	\vdots \\
	B^n_{2}
	\end{pmatrix} \right) ,\dots,\right.\\ &  \left(
	\begin{pmatrix}
	W^1_{L} & 0 & 0 & \cdots & 0 \\
	0 & W^2_{L} & 0 & \cdots & 0 \\
	0 & 0 & W^3_{L} & \cdots & 0 \\
	\vdots & \vdots & \vdots & \ddots & \vdots \\
	0 & 0 & 0 & \cdots & W^n_{L}
	\end{pmatrix},
	\begin{pmatrix}
	B^1_{L} \\
	B^2_{L} \\
	B^3_{L} \\
	\vdots \\
	B^n_{L}
	\end{pmatrix} \right) \bBigg@{5})
	\end{align}
	(cf.\ Definition \ref{def5.1}).
\end{definition}

\begin{lemma}\label{lem4.10}
	Let $a \in C(\R,\R)$ be the activation function and let $n \in \N$, $\theta = (\theta_1,\theta_2,\dots,\theta_n)\in \mathcal{N}^n$ satisfy  $\mathcal{L}(\theta_1)=\mathcal{L}(\theta_2)=\dots=\mathcal{L}(\theta_n)$ (cf.\ Definition \ref{def5.1}). Then
	\begin{enumerate}[label = \roman*)]
		\item it holds that
		\begin{align}
		\mathcal{D}(\mathbf{P}_n(\theta)) = \left(\textstyle\sum_{j=1}^{n} \mathbb{D}_0(\theta_j), \textstyle\sum_{j=1}^{n} \mathbb{D}_1(\theta_j),\dots, \textstyle\sum_{j=1}^{n} \mathbb{D}_{\mathcal{L}(\theta_1)}(\theta_j) \right)
		\end{align}
		and
		\item it holds that for all $x_1 \in \R^{\mathcal{I}(\theta_1)}, x_2 \in \R^{\mathcal{I}(\theta_2)},\dots,x_n \in \R^{\mathcal{I}(\theta_n)}$ that
		\begin{align}
		(\mathcal{R}_a(\mathbf{P}_n(\theta)))(x_1,x_2,\dots,x_n) = ((\mathcal{R}_a\theta_1)(x_1), (\mathcal{R}_a\theta_2)(x_2),\dots,(\mathcal{R}_a\theta_n)(x_n))
		\end{align}
	\end{enumerate}
	(cf.\ Definitions \ref{def5.3} and \ref{def5.7}).
\end{lemma}

The second statement in Lemma \ref{lem4.10} above shows the natural property that the realization of parallelized FNNs is the parallelization of the realizations of the FNNs and the first statement allows us to conclude the complexity bounds in Lemma \ref{sumDNN} below.

\begin{definition}[Matrix multiplications of FNNs]\label{def5.9}
	Let $L \in \N,\ l_0,l_1,\dots,l_L,n \in \N$, $\theta = ((W_1,B_1),(W_2,B_2),\dots,(W_L,B_L)) \in (\times_{k=1}^{L}( \R^{l_k \times l_{k-1}} \times \R^{l_k}))$, $\mathcal{W} \in \R^{n \times l_L}$, $\mathbb{W} \in \R^{l_0 \times n}$. Then we denote by $\mathcal{W} \circledast \theta \in \mathcal{N}$ and $\theta \circledast \mathbb{W} \in \mathcal{N}$ the FNNs given by
	\begin{align}
	\mathcal{W} \circledast \theta = ((W_1,B_1),(W_2,B_2),\dots,(W_{L-1},B_{L-1}),(\mathcal{W}W_L,\mathcal{W}B_L))
	\end{align}
	and
	\begin{align}
	\theta \circledast \mathbb{W} = ((W_1 \mathbb{W},B_1),(W_2,B_2),\dots,(W_L,B_L))
	\end{align}
	(cf.\ Definition \ref{def5.1}).
\end{definition}

The next result shows the  that the matrix multiplications of an FNN correspond to the linear projection of the output, respectively input, with the multiplied matrix.

\begin{lemma}\label{lem4.12}
	Let $a \in C(\R,\R)$  and let $\ \theta \in \mathcal{N}$ (cf.\ Definition \ref{def5.1}). Then
	\begin{enumerate}[label =\roman*)]
		\item it holds for all $m \in \N$, $\mathcal{W} \in \R^{m \times \mathcal{O}(\theta)}$ that $\mathcal{R}_a(\mathcal{W} \circledast \theta) \in C(\R^{\mathcal{I}(\theta)}, \R^m)$,
		\item it holds for all $m \in \N$, $\mathcal{W} \in \R^{m \times \mathcal{O}(\theta)},\ x \in \R^{\mathcal{I}(\theta)}$ that 
		\begin{align}
		(\mathcal{R}_a(\mathcal{W} \circledast \theta))(x) = \mathcal{W}(\mathcal{R}_a \theta(x)),
		\end{align}
		\item it holds for all $n \in \N$, $\mathbb{W} \in \R^{\mathcal{I}(\theta) \times n}$ that $\mathcal{R}_a(\theta \circledast \mathbb{W}) \in C(\R^{n}, \R^{\mathcal{O}(\theta)})$, and
		\item it holds for all $n \in \N$, $\mathbb{W} \in \R^{\mathcal{I}(\theta) \times n}, \ x \in \R^n$ that 
		\begin{align}
		(\mathcal{R}_a(\theta \circledast \mathbb{W}))(x) = \mathcal{R}_a \theta(\mathbb{W}x)
		\end{align}
	\end{enumerate}
	(cf.\ Definitions \ref{def5.3} and \ref{def5.9}).
\end{lemma}

\begin{lemma}[Sum of FNNs with same architecture]\label{sumDNN}
	Let $a \in C(\R,\R)$, $M \in \N$, $h_1,h_2, \allowbreak \dots, \allowbreak  h_M \in \R$ and let $(\theta_m)_{m \in \{1,2, \allowbreak \dots, \allowbreak M \}} \in \mathcal{N}$ satisfy that $\mathcal{D}(\theta_1)=\mathcal{D}(\theta_2)=\dots=\mathcal{D}(\theta_M)$ 	(cf.\ Definition \ref{def5.1}). Then there exists $\psi \in \mathcal{N}$ such that for all $x \in \R^{\mathcal{I}(\theta_1)}$ it holds that $\mathcal{R}_a \psi \in C(\R^{\mathcal{I}(\theta_1)},\R^{\mathcal{O}(\theta_1)})$, $\mathcal{P}(\psi) \leq M^2 \mathcal{P}(\theta_1)$, and
	\begin{align}
(	\mathcal{R}_a \psi)(x) = \textstyle\sum\limits_{m=1}^{M} h_m [(\mathcal{R}_a \theta_m)(x)]
	\end{align}
	(cf.\ Definition \ref{def5.3}).
\end{lemma}
The proof of Lemma \ref{sumDNN} follows from Lemma \ref{lem4.10} and Lemma \ref{lem4.12}, see, e.g., \cite[Lemma 5.4]{jentzen2018proof} for further details.
Lemma \ref{sumDNN} assures the existence of an FNN whose realization is equal to a weighted sum of other FNNs with same architecture. This lemma is used to obtain a similar result for ResNets in \Cref{lemparallelResNets} in Section~\ref{ResidualNetworks} below, which in turn,  allows us to apply  Monte Carlo type estimates using ResNets.

\section{Residual networks (ResNets)}\label{ResidualNetworks}

	In the following section we present definitions and results on ResNets  analogous to the ones for FNNs which we  recalled in Section \ref{FFDNNs} above. The main difference between FNNs as introduced in Section \ref{FFDNNs} and general ResNets is the shortcut connection between different layers, see He et al.~\cite{he2016deep}. There are several different architectures of ResNets, see, e.g., \cite{avelin2021neural,e2019priori,he2016deep,he2016identity,muller2020space,petersen2018optimal,weinan2018deep,zhang2017residual}. In this article, we present a definition inspired by \cite{he2016identity,muller2020space}, where ResNets are designed as a sequence of FNNs and corresponding shortcut connections. This shortcut connections are realized by linear projections of the output from the previous step (cf.\ Definition \ref{ReaResNet} below).

	\begin{definition}[ResNet]\label{DefResNet}
	We denote the set of all ResNets by
	\begin{align*}
	\ResNet  \colon= \cup_{n\in\N}\ResNet_n,
	\end{align*}
	where for every $n \in \N$,
	\begin{align*}
	\ResNet_n  \colon= 
	\cup_{\left(d_0,d_1,\dots,d_n\right)\in \N^n}
	\left\{\left(\Gamma_1,\theta_1,\Gamma_2,\theta_2,\dots,\Gamma_n,\theta_n\right)\middle\vert
	\begin{array}{cc}
	\theta_k\in \mathcal{N} \, \text{is such that}\\
\mathcal{I}(\theta_k) = d_{k-1}, \mathcal{O}(\theta_k) = d_{k},\\
\text{and }	\Gamma_k\in\R^{d_k\times d_{k-1}}\\
	\forall \, k\in\{1,2,\dots,n\}
	\end{array}
	\right\}
	\end{align*}
(cf.~\Cref{def5.1}).
	Let $n\in\N$, $d_0,d_1,d_2,\dots,d_n\in\N$, and let the ResNet $\Theta\in\ResNet$ be given by
	\begin{align*}
	\Theta &= \left(\Gamma_1,\theta_1,\Gamma_2,\theta_2,\dots,\Gamma_n,\theta_n\right),
	\end{align*}
	where for every $k \in \{1, 2, \ldots, n\}$ we have  $\mathcal{I}(\theta_k) = d_{k-1}, \mathcal{O}(\theta_k) = d_{k}$, and $\Gamma_k\in\R^{d_{k}\times d_{k-1}}$.	We call $\theta= \left(\theta_1,\theta_2,\dots,\theta_n\right)\in \mathcal{N}^n$ the \textit{residual blocks} of $\Theta$.  Moreover, we denote by $\ResLayer,\ResPar \colon\ResNet\to \N$ the functions which satisfy $\ResLayer(\Theta) = n$,  $\ResPar\left(\Theta\right) = \sum_{i=1}^n\left(\mathcal{P}(\theta_i) + d_id_{i-1}\right)$ (cf.\ Definition \ref{def5.1}). We call $\ResLayer(\Theta)$ the length of $\Theta$ and $\ResPar(\Theta)$ the complexity of $\Theta$. Furthermore, we denote by $\ResDims \colon\ResNet \to \cup_{L\in\N}{\N^L}$ the function which satisfies $\ResDims(\Theta) = (d_0,d_1,\dots,d_{\ResLayer(\Theta)})$.
	\end{definition}
	
	\begin{definition}[Realization of ResNets]\label{ReaResNet}
		Let $a\in C(\R,\R)$. Then we denote by $\ResReal_a \colon\ResNet\to \cup_{u,v\in \N} C(\R^{u},\allowbreak \R^{v})$ the function which satisfies for all
		\begin{align*}
				\Theta = \left(\Gamma_1, \theta_1, \Gamma_2, \theta_2, \dots, \Gamma_n, \theta_n \right)  \in \ResNet
		\end{align*}
 and $x_0 \in \R^{\mathcal{I}(\theta_1)}$,  $x_1\in \R^{\mathcal{I}(\theta_2)} =  \R^{\mathcal{O}(\theta_1)}$, \dots, $x_n\in \R^{\mathcal{O}(\theta_n)}$ with 
	\begin{align}\label{h.i}
	x_i = \Gamma_i x_{i-1} + \mathcal{R}_a\theta_i\left(x_{i-1}\right)
	\end{align}
	for all $i\in \{1,2,\dots,n\}$, that $(\ResReal_a  \Theta)(x_0) = x_n$  (cf.\ Definitions \ref{def5.1}, \ref{def5.3}, and \ref{DefResNet}). We call $a \in C(\R,\R)$ the activation function and the function $\ResReal_a\Theta\in C(\R^{\mathcal{I}(\theta_1)},\R^{\mathcal{O}(\theta_n)})$ the realization of the ResNet $\Theta \in \ResNet$.
	\end{definition}
	Figure \ref{figRealizationResNet} illustrates the realization of a ResNet as introduced in Definition \ref{ReaResNet}.
	
	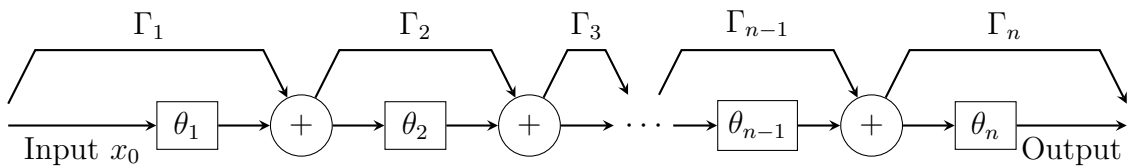
\begin{figure}[H]
		\centering
		\begin{tikzpicture}
		\tikzstyle{input} = [rectangle, minimum width=0.001cm, minimum height=0.8cm,text centered, draw=white]
		\tikzstyle{output} = [rectangle, minimum width=0.001cm, minimum height=0.8cm,text centered, draw=white]
		\tikzstyle{weights} = [rectangle, minimum width=0.8cm, minimum height=0.4cm,text centered, draw=black]
		\tikzstyle{bias} = [circle, minimum width=0.6cm, minimum height=0.4cm, text centered, draw=black]
		\tikzstyle{a} = [rectangle, minimum width=0.8cm, minimum height=0.4cm, text centered, draw=black]
		\tikzstyle{arrow} = [thick,->,>=stealth]
		
		\node (input) [input] {};
		\node (block1) [weights, right of=input, xshift=1.5cm] {$\theta_1$};
		\node (block2) [bias, right of=block1, xshift=0.5cm] {$+$};
		\node (block3) [weights, right of=block2, xshift=0.5cm] {$\theta_2$};
		\node (input2) [bias,right of=block3,xshift=0.5cm] {$+$};
		\node (dots)   [input,right of = input2,xshift=0.5cm]{$\cdots$};
		\node (blockn1) [weights, right of=dots, xshift=0.5cm] {$\theta_{n-1}$};
		\node (dots2)  [bias,right of=blockn1,xshift=0.5cm]{$+$};
		\node (blockn) [weights, right of=dots2, xshift=0.5cm] {$\theta_n$};
		\node (output) [output,right of=blockn,xshift=1cm] {};
		
		\draw [arrow] (input) --node[anchor = north]{Input $x_0$}(block1);
		\draw [arrow] (input) -- +(0.5,1) -- node[anchor= south,pos =0.5]{$\Gamma_1$} +(3.5,1) -- (block2);
		\draw [arrow] (block1) -- (block2);
		\draw [arrow] (block2) -- +(0.5,1) -- node[anchor= south,pos =0.5]{$\Gamma_2$} +(2.5,1) --(input2);
		\draw [arrow] (block2) -- (block3);
		\draw [arrow] (block3) -- (input2);
		\draw [arrow] (input2) -- (dots);
		\draw [arrow] (input2) -- +(0.5,1) -- node[anchor= south,pos =0.5]{$\Gamma_3$} +(1,1) -- (dots);
		\draw [arrow] (dots) -- (blockn1);
		\draw [arrow] (dots) -- +(0.5,1) -- node[anchor= south,pos =0.5]{$\Gamma_{n-1}$} +(2.5,1) --(dots2);
		\draw [arrow] (blockn1) -- (dots2);		
		\draw [arrow] (dots2) -- +(0.5,1) -- node[anchor= south,pos =0.5]{$\Gamma_{n}$} +(3,1) -- (output);
		\draw [arrow] (dots2) -- (blockn);
		\draw [arrow] (blockn) --node[anchor = north]{Output}(output);
		
		\end{tikzpicture}
		\caption{Realization of a ResNet.}
		\label{figRealizationResNet}
	\end{figure}
	Note that in our definition of a ResNet we only allow shortcuts over one FNN. If $d_0 = d_1 = \dots = d_n$ this definition  suits the Euler-Maruyama type approximation when setting the shortcut matrices $\Gamma_{i}$ to the identity matrices.
	
	\begin{definition}[Composition of ResNets]\label{defcompres}
			Let $n,m\in \N$ and let the ResNets $\Theta^1 = (\Gamma^1_1,\theta^1_1,\Gamma^1_2,\theta^1_2, \allowbreak \dots, \allowbreak \Gamma^1_n,\theta^1_n)$, $\Theta^2 = (\Gamma^2_1,\theta^2_1,\Gamma^2_2,\theta^2_2,\dots,\Gamma^2_m,\theta^2_m) \in \ResNet$ satisfy $\mathcal{O}(\theta^1_n) = \mathcal{I}(\theta^2_1)$ (cf.~Definitions~\ref{def5.1} and \ref{DefResNet}). Then we define  the composition of $\Theta^1$ and $\Theta^2$ by 
			\begin{align*}
				\Theta^2 \bullet \Theta^1 \colon= (\Gamma^1_1,\theta^1_1,\Gamma^1_2,\theta^1_2,\dots,\Gamma^1_n,\theta^1_n,\Gamma^2_1,\theta^2_1,\Gamma^2_2,\theta^2_2,\dots,\Gamma^2_m,\theta^2_m) \in \ResNet.
			\end{align*}
	\end{definition}
	
Note that under the condition that $\mathcal{O}(\theta^1_n)$, the output dimension of $\Theta^1$, is equal to $\mathcal{I}(\theta^2_1)$, the input dimension of $\Theta^2$, the composition $\Theta^2 \bullet \Theta^1$ is well defined and again a ResNet. The following result, \Cref{lemcomp} below,  shows the natural property that the realization of two composed ResNets is the composition of the realizations of these two ResNets.

	\begin{lemma}\label{lemcomp}
		Let $n,m\in \N$, $a\in C(\R,\R)$, and let  $\Theta^1 = (\Gamma^1_1,\theta^1_1,\Gamma^1_2,\theta^1_2, \allowbreak \dots, \allowbreak \Gamma^1_n,\theta^1_n)$, $\Theta^2 = (\Gamma^2_1,\theta^2_1,\Gamma^2_2,\theta^2_2,\dots,\Gamma^2_m,\theta^2_m) \in \ResNet$ satisfy $\mathcal{O}(\theta^1_n) = \mathcal{I}(\theta^2_1)$ (cf.~Definitions~\ref{def5.1} and \ref{DefResNet}). Then the following properties hold
		\begin{enumerate}[label = \roman*)]
			\item\label{itone} $\ResReal_a(\Theta^2 \bullet \Theta^1) = \ResReal_a(\Theta^2)\circ \ResReal_a(\Theta^1)$,
			\item\label{ittwo} $\ResPar(\Theta^2 \bullet \Theta^1) = \ResPar(\Theta^1) + \ResPar(\Theta^2)$
		\end{enumerate}
	(cf.~\Cref{ReaResNet,defcompres}).
	\end{lemma}

	\begin{proof}
		To show statement \ref{itone}, for arbitrary $x_0\in \R^{\mathcal{I}(\theta_1^1)}$ let us define $x_i  \colon= \Gamma^1_i x_{i-1} + (\mathcal{R}_a\theta^1_i)(x_{i-1})$ for all $i\in \{1,2,\dots,n\}$ and $x_j  \colon= \Gamma^2_{j-n} x_{j-1} + (\mathcal{R}_a\theta^2_{j-n})(x_{j-1})$ for all $j\in \{n+1,n+2,\dots,n+m\}$. Then
		\begin{align*}
	[	\ResReal_a(\Theta^2\bullet\Theta^1)](x_0) &= x_{n+m} = [\ResReal_a(\Theta^2)](x_n)= [\ResReal_a(\Theta^2)] [\mathscr{R}_a(\Theta^1)(x_0)]\\ &= [\ResReal_a(\Theta^2)\circ \ResReal_a(\Theta^1)](x_0).
		\end{align*}
		For statement \ref{ittwo} note that
		\begin{align*}
		\ResPar (\Theta^2 \bullet \Theta^1) &= \textstyle\sum\limits_{j=1}^m\left(\mathcal{P}(\theta_j^2) + d^2_jd^2_{j-1} \right)+ \textstyle\sum\limits_{i=1}^n\left(\mathcal{P}(\theta_i^1) + d^1_id^1_{i-1}\right)\\
		&= \ResPar(\Theta^1) + \ResPar(\Theta^2),
		\end{align*}
	where $d^1_0,d^1_1,d^1_2, \dots, d^1_n$, $d^2_0,d^2_1, \dots ,d^2_m\in\N$ satisfy $\forall \, i \in \{1, 2, \ldots, n\} \colon \mathcal{I}(\theta_i^1) = d^1_{i-1}, \mathcal{O}(\theta^1_i) = d^1_{i}$ and $\forall \, j \in \{1, 2, \ldots, m\} \colon \mathcal{I}(\theta_j^2) = d^2_{j-1}, \mathcal{O}(\theta^2_j) = d^2_{j}$.
		The proof of Lemma \ref{lemcomp} is completed.
	\end{proof}

Next we define  the composition of a ResNet with a standard FNN.

\begin{definition}[Composition of a ResNet with an FNN]\label{defcompFFREs}
	Let $n\in\N$, $\Theta = (\Gamma_1, \theta_1, \Gamma_2, \theta_2, \allowbreak \dots, \allowbreak \Gamma_n, \theta_n) \in \ResNet$, and $\phi\in\mathcal{N}$ with $\mathcal{I}(\phi) = \mathcal{O}(\theta_n)$ (cf.~Definitions~\ref{def5.1} and \ref{DefResNet}). Moreover, let $A\in \R^{\mathcal{O}(\phi)\times \mathcal{I}(\phi)}$ be the matrix with only zero entries. Then we define the composition of $\phi$ and $\Theta$ by
	\begin{align*}
	\phi \diamond \Theta  &\colon= (\Gamma_1,\theta_1,\Gamma_2,\theta_2,\dots,\Gamma_n,\theta_n,A,\phi).
	\end{align*}
\end{definition}

Observe that under the assumption that $\mathcal{O}(\theta_n)$, the output dimension of $\Theta$, is equal to $\mathcal{I}(\phi)$, the input dimension of $\phi$, the composition $\phi \diamond \Theta$ is well defined and again a ResNet. 
Similarly, to \Cref{lemcomp} above, the next result demonstrates that the realization of the composition of a ResNet with an FNN is the composition of the realizations of this ResNet and FNN.

\begin{lemma}\label{lemcompFFResnew}
	Let $n\in\N$, $a\in C(\R,\R)$, $\Theta = (\Gamma_1,\theta_1,\Gamma_2,\theta_2,\dots,\Gamma_n,\theta_n)\in \ResNet$, and $\phi\in\mathcal{N}$ with $\mathcal{I}(\phi) = \mathcal{O}(\theta_n)$ (cf.~Definitions~\ref{def5.1} and \ref{DefResNet}).  Then the following two properties hold
	\begin{enumerate}[label =\roman*)]
		\item $\ResReal_a(\phi \diamond \Theta) = \left[\mathcal{R}_a(\phi)\right]\circ \left[\ResReal_a(\Theta)\right]$,
		\item $\ResPar(\phi \diamond \Theta) = \ResPar(\Theta) + \mathcal{P}(\phi) + \mathcal{I}(\phi)\mathcal{O}(\phi)$
	\end{enumerate}
(cf.~Definitions \ref{ReaResNet}, \ref{defcompFFREs}, and \ref{def5.3}).
\end{lemma}
\begin{proof}
Let $A\in \R^{\mathcal{O}(\phi)\times \mathcal{I}(\phi)}$ be the matrix with only zero entries.
	The statements follow directly from Lemma \ref{lemcomp} and using the fact $\phi \diamond \Theta = (A,\phi)\bullet \Theta$ (cf.~\Cref{defcompres}).
\end{proof}

The above Definition \ref{defcompFFREs} and Lemma \ref{lemcompFFResnew} are important to connect FNNs to our theory about ResNets. Setting the shortcut connection to zero allows us to compose ResNets with FNNs and obtain the natural property that the realization of the composition of a ResNet and an FNN is the composition of the realization of a ResNet and an FNN.

The next definition presents the parallelization of ResNets analogous to Definition \ref{def5.7}.
	\begin{definition}[Parallelization of ResNets with same architecture]\label{DefParResnets}
		Let $u\in \N$ and $\Theta^j  \colon= (\Gamma^j_1,\theta^j_1,\Gamma^j_2,\theta^j_2,\dots, \Gamma^j_n,\theta^j_n) \in \ResNet$ for all $j\in\{1,2,\dots,u\}$ and some $n\in\N$ (cf.~Definition  \ref{DefResNet}).  Moreover, assume that $\mathcal{D}(\theta_i^1)= \dots = \mathcal{D}(\theta_i^u)$ for all $i\in \{1,2,\dots,n\} $ (cf.~Definition~\ref{def5.1}). Then we define 
		\begin{align*}
			\ResParal_u(\Theta^1,\Theta^2,\dots,\Theta^u)  \colon=
			\bBigg@{4}(&\begin{pmatrix} 
			\Gamma_1^1 & 0 & \dots & 0 \\
			0 & \Gamma_1^2 & \dots & 0\\
			\vdots & \vdots & \ddots & \vdots\\
			0  & 0 & \dots & \Gamma_1^u 
			\end{pmatrix},
			\mathbf{P}_u(\theta_1^1,\theta_1^2,\dots,\theta_1^u),\\
			&\begin{pmatrix} 
			\Gamma_2^1 & 0 & \dots & 0 \\
			0 & \Gamma_2^2 & \dots & 0\\
			\vdots & \vdots & \ddots & \vdots\\
			0  & 0 & \dots & \Gamma_2^u 
			\end{pmatrix},
			\mathbf{P}_u(\theta_2^1,\theta_2^2,\dots,\theta_2^u),
			\dots,\\
			&\begin{pmatrix} 
			\Gamma_n^1 & 0 & \dots & 0 \\
			0 & \Gamma_n^2 & \dots & 0\\
			\vdots & \vdots & \ddots & \vdots\\
			0  & 0 & \dots & \Gamma_n^u 
			\end{pmatrix},
			\mathbf{P}_u(\theta_n^1,\theta_n^2,\dots,\theta_n^u)\bBigg@{4})
		\end{align*}
		(cf.\ Definition \ref{def5.7}).
	\end{definition}

	Note that the above condition $\forall \, i\in \{1,2,\dots,n\} \colon \mathcal{D}(\theta_i^1)= \mathcal{D}(\theta_i^2) = \dots = \mathcal{D}(\theta_i^u)$ ensures that for $\ResDims(\Theta^1) = (d_0,d_1,\dots,d_{n})$ we have, for all $ i\in \{1,2,\dots, n\}$,
	\begin{align*}
	\begin{pmatrix} 
	\Gamma_i^1 & 0 & \dots & 0 \\
	0 & \Gamma_i^2 & \dots & 0\\
	\vdots & \vdots & \ddots & \vdots\\
	0  & 0 & \dots & \Gamma_i^u 
	\end{pmatrix} \in \R^{ud_i\times ud_{i-1}}
	\end{align*}
	and, by Lemma \ref{lem4.10}, we get $\mathcal{R}_a ( \mathbf{P}_u(\theta^1_i,\dots,\theta^u_i))\in C(\R^{ud_{i-1}},\R^{ud_i})$. Thus $\ResParal_u$ is well defined and is a ResNet. 
	
	The next lemma shows a natural property that the realization of a parallelized ResNets is equal to the vector of the realizations of these ResNets.
	
	\begin{lemma}\label{ResNetParallel}
		Let $n,u\in\N$,  $a\in C(\R,\R)$, let $\Theta^j  \colon= (\Gamma^j_1,\theta^j_1,\Gamma^j_2,\theta^j_2, \allowbreak\dots, \allowbreak \Gamma^j_n,\theta^j_n) \in \ResNet$, $j\in\{1,2,\dots,u\}$, satisfy  $\mathcal{D}(\theta_i^1) = \mathcal{D}(\theta_i^2) = \dots = \mathcal{D}(\theta_i^u)$ for all $i\in \{1,2,\dots,n\}$, and let $x_0^1,x_0^2, \allowbreak \dots, \allowbreak x_0^u\in \R^{\mathcal{I}(\theta_1^1)}$ (cf.~Definitions~\ref{def5.1} and \ref{DefResNet}). Then
		\begin{enumerate}[label = \roman*)]
			\item\label{ittone} $[\ResReal_a(\ResParal_u(\Theta^1,\Theta^2,\dots,\Theta^u))](x_0^1,x_0^2,\dots,x_0^u) = ((\ResReal_a\Theta^1)(x_0^1),(\ResReal_a\Theta^2)(x_0^2),\dots,(\ResReal_a\Theta^u)(x_0^u))$,
			\item\label{itttwo} $\ResPar(\ResParal_u(\Theta^1,\Theta^2,\dots,\Theta^u)) \leqslant u^2\ResPar(\Theta^1)$
		\end{enumerate}
	(cf.~Definitions \ref{ReaResNet} and \ref{DefParResnets}).
	\end{lemma}

	\begin{proof}
		Let $d_0,d_1,d_2,\dots,d_n\in\N$ satisfy for all $j\in\{1,2,\dots,u\}$ that $\ResDims(\Theta^j) = (d_0,d_1,d_2,\dots,d_n)$ and let
		 $z_0 \in \R^{ud_0}$, $z_i\in \R^{ud_i}$, $x_i^j\in\R^{d_i}$ be given for all $i\in \{1,2,\dots,n\},j\in\{1,2,\dots,u\}$ by 
		\begin{align*}
			z_0 & \colon= \begin{pmatrix}
						x_0^1\\x_0^2\\\vdots\\x_0^u
					\end{pmatrix},\\ 
					z_i & \colon= \begin{pmatrix} 
					\Gamma_i^1 & 0 & \dots & 0 \\
					0 & \Gamma_i^2 & \dots & 0\\
					\vdots & \vdots & \ddots & \vdots\\
					0  & 0 & \dots & \Gamma_i^u 
					\end{pmatrix} 
					z_{i-1} + [\mathcal{R}_a(\mathbf{P}_u(\theta_i^1,\theta_i^2,\dots,\theta_i^u))](z_{i-1}),\\
					x_i^j & \colon= \Gamma^j_i x^j_{i-1} + (\mathcal{R}_a\theta^j_i)(x^j_{i-1})
		\end{align*}
		(cf.\ Definitions \ref{def5.3} and \ref{def5.7}).
We observe that it is enough to show $z_i = (x_i^1,x_i^2,\dots,x_i^u)$ for all $i\in\{1,2,\dots,n\}$. Applying Lemma \ref{lem4.10} to every parallelized residual block yields inductively
		\begin{align*}
				z_i &= \begin{pmatrix} 
			\Gamma_i^1 & 0 & \dots & 0 \\
			0 & \Gamma_i^2 & \dots & 0\\
			\vdots & \vdots & \ddots & \vdots\\
			0  & 0 & \dots & \Gamma_i^u 
			\end{pmatrix} \!
			\begin{pmatrix} x_{i-1}^1\\x_{i-1}^2\\\vdots\\x_{i-1}^u
			\end{pmatrix} + [\mathcal{R}_a(\mathbf{P}_u(\theta_i^1,\theta_i^2,\dots,\theta_i^u))]\!\begin{pmatrix} x_{i-1}^1\\x_{i-1}^2\\\vdots\\x_{i-1}^u
			\end{pmatrix}\\
			&= \begin{pmatrix} 
			\Gamma_i^1x_{i-1}^1\\
			\Gamma_i^2x_{i-1}^2\\
			\vdots\\
			\Gamma_i^ux_{i-1}^u 
			\end{pmatrix} 
			+  \begin{pmatrix}(\mathcal{R}_a\theta_i^1)(x_{i-1}^1)\\
			(\mathcal{R}_a\theta_i^2)(x_{i-1}^2)\\
			\vdots\\
			(\mathcal{R}_a\theta_i^u)(x_{i-1}^u)
		\end{pmatrix} = \begin{pmatrix} x_{i}^1\\x_{i}^2\\\vdots\\x_{i}^u
			\end{pmatrix}.
		\end{align*}
		For statement \ref{itttwo} we use again Lemma \ref{lem4.10} and obtain
		\begin{align*}
			\ResPar(\ResParal_u(\Theta^1,\Theta^2,\dots,\Theta^u)) &= \textstyle\sum\limits_{i=1}^n \left[\mathcal{P}(\mathbf{P}_u(\theta_i^1,\theta_i^2,\dots,\theta_i^u)) + ud_iud_{i-1} \right]\\ &\leqslant u^2\bigg(\textstyle\sum\limits_{i=1}^n \left[\mathcal{P}(\theta_i^1) + d_id_{i-1} \right]\bigg) = u^2\ResPar(\Theta^1).
		\end{align*}
	\end{proof}

	A simple consequence of Lemma \ref{ResNetParallel} is the next result, \Cref{lemparallelResNets} below, about the sum of ResNets. In Section \ref{ResNetPDE}, \Cref{lemparallelResNets}  allows us to construct  ResNets representing   Monte Carlo type approximations.

	\begin{lemma}\label{lemparallelResNets}
		Let $a\in C(\R,\R)$, $u,n\in\N$, $h_1,h_2, \dots,h_u\in\R$ and let $\Theta^j  \colon= (\Gamma^j_1,\theta^j_1,\Gamma^j_2,\theta^j_2, \allowbreak \dots, \allowbreak \Gamma^j_n,\theta^j_n)\in\ResNet$, $j\in\{1,2,\dots,u\}$, satisfy $\mathcal{D}(\theta_i^1)= \mathcal{D}(\theta_i^2) = \dots = \mathcal{D}(\theta_i^u)$ for all $i\in \{1,2,\dots,n\}$, i.e., all the $i$-th residual blocks have the same architecture (cf.~Definitions~\ref{def5.1} and \ref{DefResNet}). Then there exists a ResNet $\psi\in\ResNet$ such that for all $x_0\in \R^{\mathcal{I}(\theta_1^1)}$
		\begin{align*}
		(	\ResReal_a\psi)(x_0) = \textstyle\sum\limits_{j=1}^u h_j [(\ResReal_a\Theta^j)(x_0) ]
		\end{align*}
		and $\ResPar(\psi) \leqslant u^2 \ResPar(\Theta^1)$ (cf.~Definition \ref{ReaResNet}).
	\end{lemma}

	\begin{proof}
		Let $d_0,d_1,d_2,\dots,d_n\in\N$ satisfy for all $j\in\{1,2,\dots,u\}$ that $\ResDims(\Theta^j) = (d_0,d_1,d_2,\dots,d_n)$ and set
		\begin{align*}
			\psi  \colon= \bBigg@{4}( &\begin{pmatrix} 
			\Gamma_1^1 & 0 & \dots & 0 \\
			0 & \Gamma_1^2 & \dots & 0\\
			\vdots & \vdots & \ddots & \vdots\\
			0  & 0 & \dots & \Gamma_1^u 
			\end{pmatrix}A_1,
			\mathbf{P}_u(\theta_1^1,\theta_1^2,\dots,\theta_1^u)\circledast A_1,\\
			&\begin{pmatrix} 
			\Gamma_2^1 & 0 & \dots & 0 \\
			0 & \Gamma_2^2 & \dots & 0\\
			\vdots & \vdots & \ddots & \vdots\\
			0  & 0 & \dots & \Gamma_2^u 
			\end{pmatrix},
			\mathbf{P}_u(\theta_2^1,\theta_2^2,\dots,\theta_2^u),
			\dots,\\
			&\left. A_2\begin{pmatrix} 
			\Gamma_n^1 & 0 & \dots & 0 \\
			0 & \Gamma_n^2 & \dots & 0\\
			\vdots & \vdots & \ddots & \vdots\\
			0  & 0 & \dots & \Gamma_n^u 
			\end{pmatrix}, A_2\circledast \mathbf{P}_u(\theta_n^1,\theta_n^2,\dots,\theta_n^u)\right),
		\end{align*}
		where $A_1\in\R^{ud_0\times d_0}$ and $A_2\in\R^{d_n\times ud_n}$ satisfy
		\begin{align*}
			A_1 = \begin{pmatrix}
						\Id_{d_0}\\
						\vdots\\
						\Id_{d_0}
				\end{pmatrix},\
					A_2 = \begin{pmatrix}
				h_1\Id_{d_n}&
				\dots&
				h_u\Id_{d_n}
				\end{pmatrix}
		\end{align*}
		(cf.\ Notation~\ref{def5.8} and Definitions \ref{def5.7} and \ref{def5.9}).
Moreover, let us set
		\begin{align*}
			z_0 & \colon= x_0,\\
			z_1 & \colon= \begin{pmatrix} 
			\Gamma_{1}^1 & 0 & \dots & 0 \\
			0 & \Gamma_1^2 & \dots & 0\\
			\vdots & \vdots & \ddots & \vdots\\
			0  & 0 & \dots & \Gamma_1^u 
			\end{pmatrix}A_1z_{0} +  [\mathcal{R}_a(\mathbf{P}_u(\theta_1^1,\theta_1^2,\dots,\theta_1^u))](A_1z_{0})\\
			&= \begin{pmatrix} 
			\Gamma_{1}^1 & 0 & \dots & 0 \\
			0 & \Gamma_1^2 & \dots & 0\\
			\vdots & \vdots & \ddots & \vdots\\
			0  & 0 & \dots & \Gamma_1^u 
			\end{pmatrix}\!\begin{pmatrix}
			x_0\\
			x_0\\
			\vdots\\
			x_0
			\end{pmatrix} +  [\mathcal{R}_a(\mathbf{P}_u(\theta_1^1,\theta_1^2,\dots,\theta_1^u))]\begin{pmatrix}
			x_0\\
			x_0\\
			\vdots\\
			x_0
			\end{pmatrix},\\
			z_i & \colon= \begin{pmatrix} 
			\Gamma_{i}^1 & 0 & \dots & 0 \\
			0 & \Gamma_i^2 & \dots & 0\\
			\vdots & \vdots & \ddots & \vdots\\
			0  & 0 & \dots & \Gamma_i^u 
			\end{pmatrix}z_{i-1} +  [\mathcal{R}_a(\mathbf{P}_u(\theta_i^1,\theta_i^2,\dots,\theta_i^u))](z_{i-1})
		\end{align*}
		 for all $i\in\{2,\dots,n-1\}$. Let $x_i^j\in\R^{d_i}$ with $x_i^j = \Gamma_i^j x^j_{i-1} + (\mathcal{R}_a\theta_i^j)(x^j_{i-1})$ for all $i\in\{1,2,\dots,n\}$, $j\in\{1,2,\dots,u\}$ and thus, by Lemma \ref{lem4.10} and Lemma \ref{lem4.12},
		\begin{align*}
			(\ResReal_a\psi)(x_0) &= A_2\begin{pmatrix} 
			\Gamma_n^1 & 0 & \dots & 0 \\
			0 & \Gamma_n^2 & \dots & 0\\
			\vdots & \vdots & \ddots & \vdots\\
			0  & 0 & \dots & \Gamma_n^u 
			\end{pmatrix}z_{n-1} + A_2 [\mathcal{R}_a(\mathbf{P}_u(\theta_n^1,\theta_n^2,\dots,\theta_n^u))(z_{n-1})]\\
			&= \bigg[ \textstyle\sum\limits_{j=1}^u h_j\Gamma_n^j x^j_{n-1} \bigg]+ \bigg[\textstyle\sum\limits_{j=1}^u h_j[(\mathcal{R}_a\theta_n^j)(x^j_{n-1}) ] \bigg]\\
			&= \textstyle\sum\limits_{j=1}^u h_j [\Gamma_n^j x^j_{n-1} + (\mathcal{R}_a\theta_n^j)(x^j_{n-1})]
			= \textstyle\sum\limits_{j=1}^u h_j  [(\ResReal_a \Theta^j)(x_0) ],
		\end{align*}
		where we used $(\ResReal_a \Theta^j)(x_0) = \Gamma_n^j x^j_{n-1} + (\mathcal{R}_a\theta_n^j)(x^j_{n-1})$ for all $j\in\{1,2,\cdots,u\}$. Let $\mathcal{D}(\theta_i^1)= \mathcal{D}(\theta_i^2) = \dots = \mathcal{D}(\theta_i^u) = (d_{i-1},l^i_1,l^i_2\dots,d_i)\in \N^{\mathcal{L}(\theta_i^1)}$ for all $i\in \{1,2,\dots,n\}$.
	Using Lemma \ref{lem4.10} yields that
			\begin{align*}
				\ResPar(\psi) &= \mathcal{P}(\mathbf{P}_u(\theta_1^1,\theta_1^2,\dots,\theta_1^u)\circledast A_1) + ud_1d_0 + \textstyle\sum\limits_{i=2}^{n-1} \left(\mathcal{P}(\mathbf{P}_u(\theta_i^1,\theta_i^2,\dots,\theta_i^u))+ud_{i}ud_{i-1}\right)\\ &+ \mathcal{P}(A_2\circledast \mathbf{P}_u(\theta_n^1,\theta_n^2,\dots,\theta_n^u)) + d_nud_{n-1}\\
				& = ul_1^1d_{0} + ul_1^1 + \bigg[\textstyle\sum\limits_{j=2}^{\mathcal{L}(\theta_1^1)-1}(ul_j^1ul_{j-1}^1 + ul_j^1) \bigg] + ud_1ul^1_{\mathcal{L}(\theta_1^1)-1} + ud_1\\ 
				&+ ud_1d_0 + \textstyle\sum\limits_{i=2}^{n-1} \left(\mathcal{P}(\mathbf{P}_u(\theta_i^1,\theta_i^2,\dots,\theta_i^u))+ud_{i}ud_{i-1}\right)\\
				&+ ul_1^nud_{n-1} + ul_1^n + \bigg[\textstyle\sum\limits_{j=2}^{\mathcal{L}(\theta_n^1)-1}(ul_j^nul_{j-1}^n + ul_j^n) \bigg]+ d_nul^n_{\mathcal{L}(\theta_n^1)-1} + d_n + d_nud_{n-1} \\
				&\leqslant u^2\mathcal{P}(\theta_1^1) + u^2d_1d_0 + \bigg[\textstyle\sum\limits_{i=2}^{n-1}(u^2\mathcal{P}(\theta_i^1) + u^2 d_id_{i-1}) \bigg] + u^2\mathcal{P}(\theta_n^1) + u^2d_n^1d_{n-1}\\
				&\leqslant u^2\bigg[\textstyle\sum\limits_{i=1}^n(\mathcal{P}(\theta_i^1) + d_i d_{i-1})\bigg] = u^2\ResPar(\Theta^1).
			\end{align*}
	\end{proof}

\section{ResNet approximations for Kolmogorov PDEs}\label{ResNetPDE}

	In this section, we present with the Proposition \ref{mainprop}, which is inspired by \cite[Proposition 6.1]{jentzen2018proof}, the main result of this article.   Using results from \cite{grohs2018proof,jentzen2018proof} we construct a perturbed Euler-Maruyama type approximation of a stochastic process which is linked to a solution of Kolmogorov PDE by the Feynman-Kac formula (see, e.g., \cite[Theorem 4.16 \& Corollary 4.17]{Hairer_2015} and \cite[Theorem 3.1]{jentzen2018proof}) and leads to a random field satisfying the desired weak error type approximation accuracy $\epsilon>0$. With the help of Lemma 2.1 in \cite{jentzen2018proof} we obtain the existence of a realization  of this constructed random field with the desired approximation accuracy. We then apply the results obtained in Section \ref{ResidualNetworks} above to construct a ResNet with realization equal to the realization of the random field obtained. In the last step we derive an upper bound on the complexity of this constructed ResNet which is polynomially growing in the dimension of the PDE and the reciprocal of the approximation accuracy. In Corollary \ref{maincor} we use H\"older's inequality to extend the result of Proposition~\ref{mainprop} from $p\in[2,\infty)$ to $p\in(0,\infty)$.
	\begin{setting}\label{setting2}
		Let $a\in C(\R,\R)$ be an activation function, let $T,\kappa\in(0,\infty)$ be a finite time horizon and a positive constant, respectively, for every $d \in \N$ let $A_d = (a_{d,i,j})_{i,j\in\{1,2,\dots,d\}} \in\R^{d\times d}$ symmetric and positive definite and let $\nu_d$ a probability measure on $\R^d$ w.r.t. the Borel $\sigma$-algebra. For every $d \in \N$,  the initial condition and drift coefficient, denoted  by $f_{0,d} \colon \R^d\to\R$ and $f_{1,d} \colon\R^d\to\R^d$,  satisfy for all $x,y\in\R^d$
		\begin{align*}
			\abs{f_{0,d}(x)} + \operatorname{Trace}(A_d) &\leqslant \kappa d^\kappa(1 + \norm{x}^\kappa),\\
			\norm{f_{1,d}(x) - f_{1,d}(y)} &\leqslant \kappa \norm{x-y}.
		\end{align*}
		Moreover, for every $\delta\in\left(0,1\right]$, $d\in \N$ let $\phi_\delta^{0,d},\phi_\delta^{1,d} \in \mathcal{N}$ be  FNNs with $\mathcal{R}_a\phi_\delta^{0,d} \in C(\R^d,\R)$, $\mathcal{R}_a\phi_\delta^{1,d} \in C(\R^d,\R^d)$ satisfying the growth conditions
		\begin{equation}\label{k.l}
		\begin{split}
			\abs{(\mathcal{R}_a\phi_\delta^{0,d})(x) - (\mathcal{R}_a\phi_\delta^{0,d})(y)} &\leqslant \kappa d^\kappa(1 + \norm{x}^\kappa + \norm{y}^\kappa)\norm{x-y},\\
			\norm{(\mathcal{R}_a\phi_\delta^{1,d})(x)} & \leqslant \kappa(d^\kappa+ \norm{x}),
			\end{split}
		\end{equation}
		and assume that the realizations of $(\phi_\delta^{0,d})_{\delta \in (0, 1]}$ and $(\phi_\delta^{1,d})_{\delta \in (0, 1]}$ approximate $f_{0,d}$ and $f_{1,d}$ in the sense that for all $\delta \in (0,1]$, $x\in\R^d$
		\begin{align}\label{existDNNs}
			\abs{f_{0,d}(x) - (\mathcal{R}_a\phi_\delta^{0,d})(x)} + \norm{f_{1,d}(x)- (\mathcal{R}_a\phi_\delta^{1,d})(x)} \leqslant \delta\kappa d^\kappa(1 + \norm{x}^\kappa)
		\end{align}
	(cf.\ Definitions \ref{def5.1} and \ref{def5.3}).
	\end{setting}

	We now state the main result of this article.
	
	\begin{proposition}\label{mainprop}
		Assume Setting \ref{setting2}, let $\eta\in[1,\infty )$, $p\in[2,\infty)$, and assume the families of FNNs $(\phi_\delta^{0,d})_{d \in \N, \delta \in (0, 1]} \subseteq \mathcal{N}$ and $(\phi_\delta^{1,d})_{d \in \N, \delta \in (0, 1]} \subseteq \mathcal{N}$ satisfy 
		\begin{align}\label{d.e}
		\mathcal{P}(\phi_\delta^{0,d}) + \mathcal{P}(\phi_\delta^{1,d}) \leqslant \kappa d^\kappa \delta^{-\kappa}, \qquad d \in \N, \delta \in (0, 1].
		\end{align}
		Moreover, for every $d \in \N$ assume $ \int_{\R^d}\norm{z}^{p(2\kappa + 1)}\nu_d(dz) \leqslant \eta d^\eta$ and  let $u_d \colon\left[0,T\right]\times \R^d \to \R$ be an at most polynomially growing viscosity solution of
		\begin{align}\label{KolgPDE2}
			(\partial_t u_d)(t,x) = \langle(\nabla_xu_d)(t,x),f_{1,d}(x)\rangle + \textstyle\sum\limits_{i,j=1}^d a_{d,i,j}(\partial_{x_i}\partial_{x_j}u_d)(t,x)
		\end{align}
		for all $(t,x) \in (0,T)\times \R^d$ with the initial condition $u_d(0,x) = f_{0,d}(x)$,  $x\in\R^d$. 
		Then there exist $c\in\R$ and a family of ResNets $(\psi_{d,\epsilon})_{d \in \N, \epsilon \in (0, 1]} \subseteq  \ResNet$ such that $\ResReal_a\psi_{d,\epsilon} \in C(\R^d,\R)$, $\ResPar(\psi_{d,\epsilon}) \leqslant cd^c\epsilon^{-c}$, and 
		\begin{align}\label{mainapprox}
			\left[\int_{\R^d}\abs{u_d(T,x) - (\ResReal_a\psi_{d,\epsilon})(x)}^p \, \nu_d(dx)\right]^{\frac{1}{p}} < \epsilon
		\end{align}
	for all $d \in \N$, $\epsilon \in (0, 1]$ (cf.~Definitions~\ref{DefResNet} and \ref{ReaResNet}).
	\end{proposition}
	Note that the assumptions \eqref{d.e} \& \eqref{existDNNs} imply that the functions $(f_{0,d})_{d \in \N}$ and $(f_{1,d})_{d \in \N}$ can  be approximated with FNNs without suffering the CoD. This is crucial in constructing  ResNets not suffering the CoD in approximating the solution functions $(u_d)_{d \in \N}$.

	Our proof of \Cref{mainprop} is  similar to the proof of \cite[Proposition 6.1]{jentzen2018proof}. Since the first part of our proof only uses results on  Euler-Maruyama type approximations (cf.\ \cite[Section 4]{jentzen2018proof}) and no FNN or ResNet specific results, we refer at some points to \cite{jentzen2018proof} for further details.
	\begin{proof}
		Let $\iota  \colon= \max\{\kappa,1\}$ and for every $d \in \N$, $\delta \in (0, 1]$ let $\mathcal{A}_d  \colon= \sqrt{2A_d}\in \R^{d\times d}$ and  $\Phi_\delta^{0,d} \colon \R^d\to\R$, $\Phi_\delta^{1,d} \colon\R^d\to\R^d$ be given by 
		\begin{align*}
			\Phi_\delta^{0,d}(x) = (\mathcal{R}_a\phi_\delta^{0,d})(x), \qquad \Phi_\delta^{1,d}(x) = (\mathcal{R}_a\phi_\delta^{1,d})(x), \qquad x\in\R^d.
		\end{align*}
		Let $\left(\Omega, \mathcal{F}, \mathbb{P}\right)$ be a  probability space rich enough such that we can define countably many independent standard $\mathcal{F}$-adapted Brownian motions $W^{d,m} \colon\left[0,T\right]\times \Omega\to\R^d$ for $d \in \N$, $m\in\N_0$, and define $\varpi_{d,q}  \colon= \E[\norm{\mathcal{A}_dW_T^{d,0}}^q]^{\frac{1}{q}}\in \R$ for all $q\in(0,\infty)$. Furthermore, for every $d \in \N$,  $x\in\R^d$ let $X^{d,x} \colon[0,T]\times \Omega\to\R^d$ be the process which satisfies
		\begin{align}\label{stochproc}
			X_t^{d,x} = x + \int_0^t f_{1,d}(X_s^{d,x}) \, ds + \mathcal{A}_dW_t^{d,0}, \qquad t \in [0, T].
		\end{align}
		Note that, e.g., by   \cite[Theorem 3.1]{jentzen2018proof} we get that these processes are indeed unique and satisfy
		\begin{align}\label{Feynmansolrelation}
			u_d(T,x) = \E[f_{0,d}(X_T^{d,x})], \qquad  x \in \R^d, d \in \N.
		\end{align}
		The connection (\ref{Feynmansolrelation}) between the deterministic solutions of the PDEs (\ref{KolgPDE2}) and the stochastic processes (\ref{stochproc}) leads us to the following proof structure.
		\begin{enumerate}[label =\roman*)]
			\item\label{1} In the first step we construct a family of  perturbed Euler-Maruyama type approximations of the processes in (\ref{stochproc}). 
			\item\label{2}  Using the values of these perturbed Euler-Maruyama type approximations at the final time $T$ together with the relation (\ref{Feynmansolrelation}), we can apply the weak error  estimates presented in \cite[Proposition 4.2]{jentzen2018proof} and a Monte Carlo type estimate to construct a random field which approximates (in expectation) a solution of the consider Kolmogorov PDE with a precision of at most $\epsilon>0$.
			\item\label{3}  Lemma 2.1 in \cite{jentzen2018proof} then assures the existence of a realization of the constructed random field with this desired approximation accuracy.
			\item\label{4} Using results from Section \ref{ResidualNetworks} we show that there exists a ResNet with realization equal to the obtained realization of the constructed random field.
			\item\label{5} In a last step we show that the complexity of the constructed ResNet grows at most polynomially in the dimension $d\in\N$ and the reciprocal of the approximation accuracy $\epsilon>0$.
		\end{enumerate}

		\textbf{Step \ref{1}:} To define the perturbed Euler-Maruyama type approximation we first define time discretization  $\chi_\delta \colon \allowbreak [0,T] \allowbreak\to[0,T]$   by
		\begin{align*}
			\chi_\delta(t)  \colon= \delta^2\max \! \left\lbrace  \N_{0} \cap \left[0,\tfrac{t}{\delta^2}\right] \right\rbrace, \qquad t \in [0, T], \delta \in (0, 1],
		\end{align*}
		i.e., we use a stepsize of $\delta^2$ and $\frac{\chi_\delta(t)}{\delta^2}+1$ is the number of steps made at time $t\in\left[0,T\right]$. Then we define the perturbed Euler-Maruyama type approximation $Y^{\delta,d,m,x} \colon\left[0,T\right]\times\Omega\to\R^d$ iteratively by
		\begin{align}\label{a.b}
			Y_t^{\delta,d,m,x}  \colon= x + \int_0^t\Phi_\delta^{1,d}(Y_{\chi_\delta(s)}^{\delta,d,m,x}) \, ds + \mathcal{A}_dW_t^{d,m}
		\end{align}
		for all  $\delta \in (0, 1]$, $d,m\in\N, x\in\R^d,t\in\left[0,T\right]$ (cf.\ \cite[Section 4]{jentzen2018proof}). Note that we consider the Euler-Maruyama type approximation with the drift coefficient $\Phi_\delta^{1,d}$ instead of $f_{1,d}$ and call it perturbed Euler-Maruyama type approximation, because we approximate $f_{1,d}$  with the FNN realizations $(\Phi_\delta^{1,d})_{\delta \in (0, 1]}$ which may, in general, differ
		from $f_{1,d}$. However, the assumption (\ref{existDNNs}) ensures that the functions $(\Phi_\delta^{1,d})_{\delta \in (0, 1]}$ approximate $f_{1,d}$ well in the sense of (\ref{existDNNs}).

		\textbf{Step \ref{2}:}  \cite[Lemmas 4.1 \& 4.2]{jentzen2018proof}, (\ref{Feynmansolrelation}), the weak error estimates presented in \cite[Proposition 4.2]{jentzen2018proof}, and a Monte Carlo type estimate (cf.\ \cite[Corollary 2.5]{grohs2018proof}) yield for every $d \in \N$, $\varepsilon \in (0, 1]$ the existence of the random field 
		\begin{align}\label{randf}
		\frac{1}{\mathfrak{M}_{d,\epsilon}}\bigg[ \textstyle\sum\limits_{m = 1}^{\mathfrak{M}_{d,\epsilon}}\Phi_{\mathfrak{D}_{d,\epsilon}}^{0,d}\!\left( Y_T^{\mathfrak{D}_{d,\epsilon},d,m, \cdot} \right) \! \bigg] \colon \R^d \times \Omega \to \R
		\end{align} 
		satisfying
		\begin{align}\label{expapprox}
			\int_{\R^d}\E \bigg[ \bigg| u_d(T,x) - \frac{1}{\mathfrak{M}_{d,\epsilon}}\bigg[ \textstyle\sum\limits_{m = 1}^{\mathfrak{M}_{d,\epsilon}}\Phi_{\mathfrak{D}_{d,\epsilon}}^{0,d} \! \left( Y_T^{\mathfrak{D}_{d,\epsilon},d,m,x} \right)\! \bigg] \bigg|^p  \bigg] \, \nu_d(dx) < \epsilon^p,
		\end{align}
		(for more details see \cite[pages 1193-1198]{jentzen2018proof}). In this step we used $\mathfrak{M}_{d,\epsilon}$ independent perturbed Euler-Maruyama type approximations with step size $\mathfrak{D}_{d,\epsilon}^2$, where
		\begin{align*}
			\mathfrak{M}_{d,\epsilon}  \colon= \min\!\left\lbrace n\in\N \colon \left(\tfrac{2^{(\kappa+4)}p\kappa d^\kappa e^{\kappa^2T}}{\epsilon}\right)^2  \!\left[ 1 + \left( \kappa d^\kappa T + \sqrt{2(p\iota-1)\kappa d^\kappa T} \right)^{p\kappa} + \eta d^\eta \right]^{\frac{2}{p}} \leqslant n \right\rbrace 
		\end{align*}
		and
		\begin{align*}
		\mathfrak{D}_{d,\epsilon}  \colon= \epsilon\!\left[ \max\{2\kappa d^\kappa,1\} + T^{-\frac{1}{2}} \right]^{-1} e^{-(3 + 3\kappa + \left[\kappa^2 + 2\kappa\iota +  2\right]T)} \abs{ \max\left\lbrace 1,2\kappa(\kappa + 1)d^\kappa\right\rbrace  }^{-1}2^{-(2\iota+1)}\\
		\left[\abs{2 + \max\{1,\kappa d^\kappa,\norm{f_{1,d}(0)}\}\max\{1,T\} + \sqrt{2(2\iota - 1)\kappa d^\kappa T}}^{p\iota + p\kappa} + \eta d^\eta \right]^{-\frac{1}{p}}.
		\end{align*}
		Note that $\mathfrak{D}_{d,\epsilon}\leqslant 1$ for all $d\in\N$, $\epsilon\in\left(0,1\right]$.

		\textbf{Step \ref{3}:} Lemma 2.1 in \cite{jentzen2018proof} together with (\ref{expapprox}) assures that there exists (a not necessarily unique) $\mathfrak{w}_{d,\epsilon}\in \Omega$ such that the realization of the random field (\ref{randf}) satisfies
		\begin{align}\label{detapprox}
			\int_{\R^d}\bigg| u_d(T,x) - \frac{1}{\mathfrak{M}_{d,\epsilon}}\bigg[ \textstyle\sum\limits_{m = 1}^{\mathfrak{M}_{d,\epsilon}}\Phi_{\mathfrak{D}_{d,\epsilon}}^{0,d}\!\left( Y_T^{\mathfrak{D}_{d,\epsilon},d,m,x}(\mathfrak{w}_{d,\epsilon}) \right) \! \bigg] \bigg|^p \, \nu_d(dx) < \epsilon^p.
		\end{align}
		This gives us the desired approximation accuracy. It remains to show the existence of a ResNet $\Psi$ with realization equal to the realization of the random field in (\ref{detapprox}) and with at most polynomially growing complexity in the dimension $d\in\N$ and the reciprocal of the approximation accuracy $\epsilon>0$, i.e.,
		\begin{align}\label{parallelM}
		(	\ResReal_a\Psi)(x) = \frac{1}{\mathfrak{M}_{d,\epsilon}}\bigg[ \textstyle\sum\limits_{m = 1}^{\mathfrak{M}_{d,\epsilon}}\Phi_{\mathfrak{D}_{d,\epsilon}}^{0,d}\!\left( Y_T^{\mathfrak{D}_{d,\epsilon},d,m,x} (\mathfrak{w}_{d,\epsilon})\right) \! \bigg]
		\end{align}
		and $\ResPar(\Psi) \leqslant cd^c\epsilon^{-c}$ for some $c\in\R$ independent of $d\in \N$, $\varepsilon \in (0, 1]$.
		
		\textbf{Step \ref{4}:}  For a fixed $d  \in \N$, $\delta \in (0, 1]$ let $L\in\N$, $l_0,l_1,\dots,l_L$ such that the FNN $\phi_\delta^{1,d}\in\mathcal{N}$ approximating the drift function is given by
		\begin{align*}
			\phi_\delta^{1,d} = ((W_1,B_1),(W_2,B_2),\dots,(W_L,B_L))\in \left(\times_{k=1}^L(\R^{l_k\times l_{k-1}}\times \R^{l_k})\right).
		\end{align*}
		To define $\phi_{\delta,d,m,i,\omega}$ for all $i\in \{1,2,\dots,\frac{\chi_\delta(T)}{\delta^2} \}$, $m\in\N$, and $\omega\in\Omega$ we slightly modify $\phi_\delta^{1,d}$ by multiplying the last weight matrix $W_L$ and bias vector $B_L$ with the step size $\delta^2$ and add a Brownian motion term to the bias. This ensures that $\mathcal{R}_a(\phi_{\delta,d,m,i,\omega})$ performs the drift and diffusion of one step in the perturbed Euler-Maruyama type approximation, i.e., we define
		\begin{align}\label{e.f}
			\phi_{\delta,d,m,i,\omega}  \colon= &\left((W_1,B_1),(W_2,B_2),\dots,\left(\delta^2 W_L,\delta^2 B_L + \mathcal{A}_d \! \left[W_{\delta^2 i}^{d,m}(\omega) - W_{\delta^2 (i-1)}^{d,m}(\omega)\right]\right)\right)
			\end{align}
			 for  $i\in \{1,2,\dots,\frac{\chi_\delta(t)}{\delta^2} \}$ and for the last step 
			\begin{align}
			\phi_{\delta,d,m,\frac{\chi_\delta(T)}{\delta^2}+1,\omega}  \colon= &\Big((W_1,B_1),(W_2,B_2),\dots,\\
			&\left.\left((T-\chi_\delta(T)) W_L,(T-\chi_\delta(T)) B_L + \mathcal{A}_d \! \left[W_{T}^{d,m}(\omega) - W_{\chi_\delta(T)}^{d,m}(\omega) \right]\right)\right).\nonumber
		\end{align}
		Thus, 
		\begin{align}\label{resblocks}
		(	\mathcal{R}_a\phi_{\delta,d,m,i,\omega})(y)\nonumber &= \delta^2\Phi_\delta^{1,d}(y) + \mathcal{A}_d \! \left[W_{i\delta^2}^{d,m}(\omega) - W_{(i-1)\delta^2}^{d,m}(\omega) \right]\\ &= \int_{(i-1)\delta^2}^{i\delta^2}\Phi_\delta^{1,d}(y) \, ds + \mathcal{A}_d \! \left[W_{i\delta^2}^{d,m}(\omega) - W_{(i-1)\delta^2}^{d,m}(\omega)\right]
		\end{align}
		for all $y\in \R^{d}$, $i\in \{1,2,\dots,\frac{\chi_\delta(T)}{\delta^2} \}$ and
		\begin{align}\label{resblockend}
			\bigl(\mathcal{R}_a\phi_{\delta,d,m,\frac{\chi_\delta(T)}{\delta^2}+1,\omega} \bigr) (y)\nonumber &= (T-\chi(T))\Phi_\delta^{1,d}(y) + \mathcal{A}_d \! \left[W_{T}^{d,m}(\omega) - W_{\chi_\delta(T)}^{d,m}(\omega)\right]\\  &= \int_{\chi_\delta(T)}^{T}\Phi_\delta^{1,d}(y) \,  ds + \mathcal{A}_d \! \left[W_{T}^{d,m}(\omega) - W_{\chi_\delta(T)}^{d,m}(\omega)\right]
		\end{align}
		for all $y\in \R^{d}$. Note that this modification does not change the complexity, i.e., $\mathcal{P}(\phi_{\delta,d,m,i,\omega}) = \mathcal{P}(\phi_\delta^{1,d})$ for all $m\in\N$, $i\in \{1,2,\dots,\frac{\chi_\delta(T)}{\delta^2} \}$, and $\omega\in\Omega$.

		Further, note that $\mathcal{I}(\phi_{\delta,d,m,i+1,\omega}) = d = \mathcal{O}(\phi_{\delta,d,m,i,\omega})$ for $i \in \{ 1,2,\dots,\frac{\chi_\delta(T)}{\delta^2}\}$. This allows us to take all the shortcut matrices to be equal to the identity matrix. Using $\phi_{\delta,d,m,i,\omega}$ as residual blocks we define the ResNet $\psi_{\delta,d,m,T,\omega}\in\ResNet$ by
		\begin{align}\label{defEMrealization}
			\psi_{\delta,d,m,T,\omega}  \colon= \left(\Id_d, \phi_{\delta,d,m,1,\omega}, \Id_d, \phi_{\delta,d,m,2,\omega}, \dots, \Id_d, \phi_{\delta,d,m,\frac{\chi_\delta(T)}{\delta^2}+1,\omega}\right)
		\end{align}
		with the realization equal to the perturbed Euler-Maruyama type approximation
		\begin{align}\label{EMrealization}
		(	\ResReal_a\psi_{\delta,d,m,T,\omega})(x) = x + \int_0^T\Phi_\delta^{1,d}(Y_{\chi_\delta(s)}^{\delta,d,m,x}) \, ds + \mathcal{A}_dW_T^{d,m}  = Y_T^{\delta,d,m,x}(\omega)
		\end{align}
		(cf.\ Definition \ref{ReaResNet}, (\ref{a.b}),(\ref{resblocks}), and (\ref{resblockend})). The complexity of $\psi_{\delta,d,m,T,\omega}$ is then given by
		\begin{align}\label{b.c}
			\ResPar(\psi_{\delta,d,m,T,\omega}) = \textstyle\sum\limits_{i=1}^{\frac{\chi_\delta(T)}{\delta^2}+1}\left(\mathcal{P}(\phi_{\delta,d,m,i,\omega}) + d^2\right) = (\mathcal{P}(\phi_\delta^{1,d}) + d^2) \!\left[ \tfrac{\chi_\delta(T)}{\delta^2}+1 \right].
		\end{align}
		Composing the ResNet $\psi_{\delta,d,m,T, \omega}\in\ResNet$ with the FNN $\phi_\delta^{0,d}\in\mathcal{N}$ according to Definition \ref{defcompFFREs} we obtain, by Lemma \ref{lemcompFFResnew},  the ResNet
		\begin{align*}
			\varphi_{\delta,d,m,T,\omega} & \colon= \phi_\delta^{0,d} \diamond \psi_{\delta,d,m,T,\omega} \in\ResNet,\\
			(\ResReal_a\varphi_{\delta,d,m,T,\omega})(x) &= [( \mathcal{R}_a\phi_\delta^{0,d})\circ (\ResReal_a\psi_{\delta,d,m,T,\omega}) ](x) = \Phi_\delta^{0,d}(Y_T^{\delta,d,m,x}(\omega))
			\end{align*}
			for all $x\in\R^{d}$, $\omega\in\Omega$, with the length $(\frac{\chi_\delta(T)}{\delta^2} + 2)$ and  complexity
			\begin{align}\label{c.d}
				\begin{split}
			\ResPar(\varphi_{\delta,d,m,T,\omega}) &= \ResPar(\psi_{\delta,d,m,T,\omega}) + \mathcal{P}(\phi_\delta^{0,d}) + d \\
			&= (\mathcal{P}(\phi_\delta^{1,d}) + d^2) \!\left[ \tfrac{\chi_\delta(T)}{\delta^2}+1 \right] + \mathcal{P}(\phi_\delta^{0,d}) + d.
		\end{split}
		\end{align}
		To construct $\Psi$ as in (\ref{parallelM}), we apply Lemma \ref{lemparallelResNets} with $a\curvearrowleft a$, $u\curvearrowleft \mathfrak{M}_{d,\epsilon}$, $n\curvearrowleft (\frac{\chi_\delta(T)}{\delta^2} + 2)$, $h_1,h_2,\dots h_u\curvearrowleft \mathfrak{M}_{d,\epsilon}^{-1}$, and $\Theta^m\curvearrowleft \varphi_{\mathfrak{D}_{d,\epsilon},d,m,T,\mathfrak{w}_{d,\epsilon}}$ for ${1\leqslant m\leqslant\mathfrak{M}_{d,\epsilon}}$. This assures the existence of a ResNet $\Psi_{\epsilon,d,\mathfrak{w}_{d,\epsilon}}\in\ResNet$ such that
		\begin{align}\label{f.g}
		(	\ResReal_a\Psi_{\epsilon,d,\mathfrak{w}_{d,\epsilon}})(x) = \frac{1}{\mathfrak{M}_{d,\epsilon}}\bigg[ \textstyle\sum\limits_{m = 1}^{\mathfrak{M}_{d,\epsilon}}\Phi_{\mathfrak{D}_{d,\epsilon}}^{0,d}\!\left( Y_T^{\mathfrak{D}_{d,\epsilon},d,m,x}(\mathfrak{w}_{d,\epsilon}) \right)\! \bigg]
		\end{align}
		with the complexity estimate
	\begin{align}\label{g.h}
		\ResPar(\Psi_{\epsilon,d,\mathfrak{w}_{d,\epsilon}}) \leqslant& \mathfrak{M}_{d,\epsilon}^2\ResPar(\varphi_{\mathfrak{D}_{d,\epsilon},d,1,T,\mathfrak{w}_{d,\epsilon}}).
	\end{align}
		Combined with (\ref{detapprox}) this yields
		\begin{align}
			\label{eq:Psi:error}
			\left[\int_{\R^d}\abs{u_d(T,x) - (\ResReal_a\Psi_{\epsilon,d,\mathfrak{w}_{d,\epsilon}})(x)}^p \, \nu_d(dx)\right]^{\frac{1}{p}} < \epsilon.
		\end{align}

			\textbf{Step \ref{5}:} Using (\ref{g.h}), (\ref{c.d}), the fact that $\mathfrak{D}_{d,\epsilon}\leqslant 1$, the assumption (\ref{d.e}), the inequality $\kappa d^\kappa + d^2 + d \leqslant 3\iota d^{2\iota}$, and the bounds on $\mathfrak{M}_{d,\epsilon}$, $\mathfrak{D}_{d,\epsilon}$ derived in \cite[(6.33) \& (6.37)]{jentzen2018proof} which are given by
			\begin{align}
			\mathfrak{M}_{d,\epsilon} \leqslant 2^{2(\kappa + 4)}p^2\iota^2e^{2\kappa^2T}	\left[2 + \abs{2p\iota \max\{1,T\}}^{p\kappa} + \eta\right]d^{p\kappa \iota + \eta + 2\kappa}\epsilon^{-2}
			\end{align}
			and
			\begin{align}
			\begin{split}
			\mathfrak{D}_{d,\epsilon} \geqslant &\min\{1,\sqrt{T}\}e^{-\left(3\iota^2+3\right)\left(T+1\right)}\iota^{-3}2^{-\left(2\iota+5\right)}\\
			&\cdot \left(\left[6\iota\max\{1,T\}\right]^{p\iota + p\kappa} + \eta\right)^{-\frac{1}{p}} d^{-\left(2\kappa + \kappa(\kappa + \iota)  + \eta\right)}\epsilon,
			\end{split}
			\end{align}
			yields to the following upper bound on the complexity of $\Psi_{\epsilon,d,\mathfrak{w}_{d,\epsilon}}$ 
				\begin{equation}\label{maincompl}
			\begin{split}
			 \ResPar(\Psi_{\epsilon,d,\mathfrak{w}_{d,\epsilon}}) \leqslant  &\mathfrak{M}_{d,\epsilon}^2 \left(\left(\mathcal{P}(\phi_{\mathfrak{D}_{d,\epsilon}}^{1,d})+d^2\right) \!\left[ \tfrac{T}{\mathfrak{D}_{d,\epsilon}^2}+1 \right] + \mathcal{P}(\phi_{\mathfrak{D}_{d,\epsilon}}^{0,d}) + d\right)\\
			\leqslant& 
			\mathfrak{M}_{d,\epsilon}^2 \left(\mathcal{P}(\phi_{\mathfrak{D}_{d,\epsilon}}^{1,d}) + \mathcal{P}(\phi_{\mathfrak{D}_{d,\epsilon}}^{0,d}) + d^2 + d\right) \!\left[ \tfrac{T}{\mathfrak{D}_{d,\epsilon}^2}+1 \right]\\
			\leqslant& \mathfrak{M}_{d,\epsilon}^2 \left(\kappa d^\kappa + d^2+d\right) \mathfrak{D}_{d,\epsilon}^{-\kappa}\left[ \tfrac{T}{\mathfrak{D}_{d,\epsilon}^2}+1 \right]\\
			\leqslant& 
			\mathfrak{M}_{d,\epsilon}^2 3\iota d^{2\iota} \left[ T+1 \right] \mathfrak{D}_{d,\epsilon}^{-\kappa-2}\\
			\leqslant&
			\left( 2^{2(\kappa + 4)}p^2\iota^2e^{2\kappa^2 T}\left[2 + \abs{2p\iota \max\{1,T\}}^{p\kappa} + \eta\right] d^{p\kappa\iota + \eta + 2\kappa}\epsilon^{-2} \right)^2\\ 
			&\cdot 3\iota d^{2\iota} \left[T + 1\right] \left[ \min\{1,\sqrt{T} \}e^{-(3\iota^2 + 3)(T + 1)} \iota^{-3} 2^{-(2\iota + 5)}\right.\\
			&\left. \cdot \left(\left[6\iota \max\{1,T\}\right]^{p\iota + p\kappa} + \eta\right)^{-\frac{1}{p}} d^{-(\kappa(2 + \kappa + \iota) + \eta)} \epsilon \right]^{-\kappa-2}\\
			=& 
			\left( 2^{2(\kappa + 4)}p^2\iota^2e^{2\kappa^2 T}\left[2 + \abs{2p\iota \max\{1,T\}}^{p\kappa} + \eta\right]\right)^2 3\iota \left[T + 1\right]\\
			&\left[ \min\{1,\sqrt{T}\}e^{-(3\iota^2 + 3)(T + 1)} \iota^{-3} 2^{-(2\iota + 5)} \cdot \left(\left[6\iota \max\{1,T\}\right]^{p\iota + p\kappa} + \eta\right)^{-\frac{1}{p}} \right]^{-\kappa-2}\\ 
			&\cdot d^{2(p\kappa\iota + \eta + 2\kappa + \iota) + (\kappa(2+\kappa + \iota) + \eta)(\kappa + 2)}\epsilon^{-\kappa -6}.
			\end{split}
		\end{equation}
		This together with \eqref{eq:Psi:error} concludes the proof of  \Cref{mainprop}.
	\end{proof}

	\begin{corollary}\label{maincor}
		Assume Setting \ref{setting2}, let $\eta\in[1,\infty )$, $p\in(0,\infty)$, assume the families of FNNs $(\phi_\delta^{0,d})_{d \in \N, \delta \in (0, 1]} \subseteq \mathcal{N}$ and $(\phi_\delta^{1,d})_{d \in \N, \delta \in (0, 1]} \subseteq \mathcal{N}$ satisfy 
		\begin{align}\label{d.ecor}
		\mathcal{P}(\phi_\delta^{0,d}) + \mathcal{P}(\phi_\delta^{1,d}) \leqslant \kappa d^\kappa \delta^{-\kappa}, \qquad d \in \N, \delta \in (0, 1].
		\end{align}
		Moreover, for every $d \in \N$ assume $ \int_{\R^d}\norm{z}^{\max\{2,p\}(2\kappa + 1)}\nu_d(dz) \leqslant \eta d^\eta$ and let $u_d \colon [0,T]\times \R^d \to \R$ be an at most polynomially growing viscosity solution of
		\begin{align}\label{KolgPDE2cor}
		(\partial_t u_d)(t,x) = \langle(\nabla_xu_d)(t,x),f_{1,d}(x)\rangle + \textstyle\sum\limits_{i,j=1}^d a_{d,i,j}(\partial_{x_i}\partial_{x_j}u_d)(t,x)
		\end{align}
		for all $(t,x) \in (0,T)\times \R^d$ with initial condition $u_d(0,x) = f_{0,d}(x)$, $x\in\R^d$.
	Then there exist $c\in\R$ and a family of ResNets $(\psi_{d,\epsilon})_{d \in \N, \epsilon \in (0, 1]} \subseteq  \ResNet$ such that $\ResReal_a\psi_{d,\epsilon} \in C(\R^d,\R)$, $\ResPar(\psi_{d,\epsilon}) \leqslant cd^c\epsilon^{-c}$, and 
		\begin{align}\label{mainapproxcor}
		\left[\int_{\R^d}\abs{u_d(T,x) - (\ResReal_a\psi_{d,\epsilon})(x)}^p \, \nu_d(dx)\right]^{\frac{1}{p}} < \epsilon
		\end{align}
	for all $d \in \N$, $\epsilon \in (0, 1]$ (cf.~Definitions~\ref{DefResNet} and \ref{ReaResNet}).
		\end{corollary}
		\begin{proof}
			 If $p\in [2,\infty)$, the statement follows directly from  \Cref{mainprop}. If $p\in(0,2)$, we have $\forall \, d \in \N \colon \int_{\R^d}\norm{z}^{2(2\kappa + 1)}\nu_d(dz) \leqslant \eta d^\eta$ and, hence, get by  \Cref{mainprop} (with $p\curvearrowleft 2$ in the notation of  \Cref{mainprop}) the existence
			  of $c\in\R$ and a family of ResNets $(\psi_{d,\epsilon})_{d \in \N, \epsilon \in (0, 1]} \subseteq  \ResNet$ such that $\ResReal_a\psi_{d,\epsilon} \in C(\R^d,\R)$, $\ResPar(\psi_{d,\epsilon}) \leqslant cd^c\epsilon^{-c}$, and 
	\begin{align*}
	\left[\int_{\R^d}\abs{u_d(T,x) - (\ResReal_a\psi_{d,\epsilon})(x)}^2 \, \nu_d(dx)\right]^{\frac{1}{2}} < \epsilon
\end{align*}
			 for all $d \in \N$, $\epsilon \in (0, 1]$.
 Applying H\"older's inequality for $p\in\left(0,2\right)$ therefore yields
			\begin{align*}
				\left[\int_{\R^d}\abs{u_d(T,x) - (\ResReal_a\psi_{d,\epsilon})(x)}^p \,  \nu_d(dx)\right]^{\frac{1}{p}} \leqslant \left[\int_{\R^d}\abs{u_d(T,x) - (\ResReal_a\psi_{d,\epsilon})(x)}^2 \, \nu_d(dx)\right]^{\frac{1}{2}} < \epsilon,
			\end{align*}
			what concludes the proof of Corollary \ref{maincor}.
			\end{proof}

\section{Conclusion}\label{conc}

	We proved the existence of ResNets which are able to approximate solutions of Kolmogorov PDEs with constant diffusion  and possibly nonlinear drift coefficients without suffering the CoD. Employing results  of \cite{jentzen2018proof}, in Proposition \ref{mainprop} above (and its slight extension - Corollary \ref{maincor}) we establish  upper bounds on the complexity of suitable ResNets which grows polynomially in the dimension of the PDEs and the reciprocal of the approximation accuracy.

	One of the tools we use to prove \Cref{mainprop} is the Feynman-Kac formula, that we apply to obtain  stochastic processes which are linked to solutions of the considered Kolmogorov PDEs and  which we approximate with perturbed Euler-Maruyama type approximations. Employing these approximations we construct a random field which gives us   the desired approximation accuracy. We establish this accuracy via 
 weak error   and  Monte Carlo type estimates. Lemma 2.1 in \cite{jentzen2018proof} then ensures the existence of a realization of this random field with the same desired approximation accuracy. Finally, we  construct a ResNet with the realization equal to this realization of the  random field and which has  a polynomially growing complexity.

	In contrast to \cite{jentzen2018proof}, we use ResNets  instead of standard FNNs for approximating solutions of Kolmogorov PDEs. Since the architecture of ResNets  suits the Euler-Maruyama type approximation, there is no need for a construction of one Euler-Maruyama type approximation step as in \cite[Proposition 5.2]{jentzen2018proof} and therefore, we do not need to assume that the identity map can be described by an FNN. 	This, in turn,   enlarges the set of applicable activation functions.  Thus, the only assumption involving the activation function is the existence of FNNs approximating the drift  and the initial conditions  which do not suffer  the CoD (cf.\ (\ref{existDNNs})) and satisfy the growth conditions in (\ref{k.l}).

	Similarly to FNNs, there are  different definitions of ResNets, see e.g., \cite{avelin2021neural,e2019priori,he2016deep,he2016identity,muller2020space,petersen2018optimal,zhang2017residual}. We work with a definition of ResNets inspired by \cite{he2016identity,muller2020space} (see also \cite{avelin2021neural} for discussion), which allows us to process the signal in the next residual block without applying the activation function before. 
This and  our definition of the complexity measure enables us to construct the composition of two ResNets with the complexity being equal to the sum of the complexity of the two ResNets  (cf.\ \Cref{lemcomp} above), which is an important ingredient in obtaining polynomial complexity bounds. Using other definitions (see, e.g., \cite{he2016deep,zhang2017residual}) might lead to not being able to directly obtain these kind of polynomial complexity bounds. This problem could be resolved, for example, by plugging in the artificial identity as a residual block analogously to \cite{jentzen2018proof}. However, it would require the existence of an FNN  representing the identity map, which  significantly shrinks the set of applicable activation functions.

	As a further step, one could investigate space-time approximations of Kolmogorov PDEs using ResNets. We note that there are various articles on space-time approximations of differential equations based on FNNs or ResNets, see e.g., \cite{hornung2020spacetime,muller2020space,grohs2019spacetime}. Other extensions of the presented theory would be an analysis of recurrent and recursive ResNets as in \cite{qin2019data} in approximating solutions of PDEs.
	Some further directions  could be a search for  optimal  complexity bounds and a consideration of   complexity measures which incorporate rank constraints of the weight matrices of neural networks. 

\subsection*{Acknowledgments}

DS has been funded in part by  ETH Foundations of Data Science (ETH-FDS). 

\pagebreak

\bibliographystyle{acm}
\phantomsection
\addcontentsline{toc}{section}{References}
\bibliography{refs}

\begin{thebibliography}{10}

\bibitem{doi:10.1142/S0129065709002130}
{\sc Asaduzzaman, M., Shahjahan, M., and Murase, K.}
\newblock Faster training using fusion of activation functions for feed forward
  neural networks.
\newblock {\em Int J Neural Syst. 19}, 06 (2009), 437--448.

\bibitem{avelin2021neural}
{\sc Avelin, B., and Nystr{\"o}m, K.}
\newblock Neural {ODE}s as the deep limit of {R}es{N}ets with constant weights.
\newblock {\em Anal.\ Appl. 19}, 03 (2021), 397--437.

\bibitem{beck2018solving}
{\sc Beck, C., Becker, S., Grohs, P., Jaafari, N., and Jentzen, A.}
\newblock Solving the {K}olmogorov {PDE} by means of deep learning.
\newblock {\em J. Sci. Comput. 88}, 3 (2021), Paper No. 73, 28.

\bibitem{beck2019machine}
{\sc Beck, C., E, W., and Jentzen, A.}
\newblock Machine learning approximation algorithms for high-dimensional fully
  nonlinear partial differential equations and second-order backward stochastic
  differential equations.
\newblock {\em J.\ Nonlinear Sci. 29}, 4 (2019), 1563--1619.

\bibitem{e2019priori}
{\sc E, W., Ma, C., and Wang, Q.}
\newblock A {P}riori {E}stimates of the {P}opulation {R}isk for {R}esidual
  {N}etworks.
\newblock {\em arXiv:1903.02154\/} (2019), 19 pages.

\bibitem{weinan2018deep}
{\sc E, W., and Yu, B.}
\newblock The deep {R}itz method: a deep learning-based numerical algorithm for
  solving variational problems.
\newblock {\em Commun. Math. Stat. 6}, 1 (2018), 1--12.

\bibitem{elbrachter2020dnn}
{\sc Elbr{\"a}chter, D., Grohs, P., Jentzen, A., and Schwab, C.}
\newblock D{NN} expression rate analysis of high-dimensional {PDE}s:
  {A}pplication to option pricing.
\newblock {\em Constr.\ Approx.\/} (2021), 1--69.

\bibitem{gribonval2021approximation}
{\sc Gribonval, R., Kutyniok, G., Nielsen, M., and Voigtlaender, F.}
\newblock Approximation spaces of deep neural networks.
\newblock {\em Constr.\ Approx.\/} (2021), 1--109.

\bibitem{grohs2020deep}
{\sc Grohs, P., and Herrmann, L.}
\newblock Deep neural network approximation for high-dimensional elliptic
  {PDE}s with boundary conditions.
\newblock {\em IMA J.\ Numer.\ Anal.\/} (2021).

\bibitem{grohs2018proof}
{\sc Grohs, P., Hornung, F., Jentzen, A., and von Wurstemberger, P.}
\newblock A proof that artificial neural networks overcome the curse of
  dimensionality in the numerical approximation of {B}lack-{S}choles partial
  differential equations.
\newblock {\em arXiv:1809.02362\/} (2018), 124 pages.

\bibitem{grohs2019spacetime}
{\sc Grohs, P., Hornung, F., Jentzen, A., and Zimmermann, P.}
\newblock Space-time error estimates for deep neural network approximations for
  differential equations.
\newblock {\em arXiv:1908.03833\/} (2019), 86 pages.

\bibitem{GrohsJentzenSalimova2019}
{\sc Grohs, P., Jentzen, A., and Salimova, D.}
\newblock {Deep neural network approximations for Monte Carlo algorithms}.
\newblock {\em Accepted in Springer Nat.\ Part.\ Diff.\ Equ.\ Appl.,
  arXiv:1908.10828\/} (2019), 45 pages.

\bibitem{10.1007/978-3-319-66471-2_8}
{\sc Gruber, I., Hlav{\'a}{\v{c}}, M., {\v{Z}}elezn{\'y}, M., and Karpov, A.}
\newblock Facing {F}ace {R}ecognition with {R}es{N}et: {R}ound {O}ne.
\newblock In {\em Interactive Collaborative Robotics\/} (Cham, 2017),
  A.~Ronzhin, G.~Rigoll, and R.~Meshcheryakov, Eds., Springer International
  Publishing, pp.~67--74.

\bibitem{Hairer_2015}
{\sc Hairer, M., Hutzenthaler, M., and Jentzen, A.}
\newblock Loss of regularity for {K}olmogorov equations.
\newblock {\em Ann. Probab. 43}, 2 (2015), 468--527.

\bibitem{han2018solving}
{\sc Han, J., Jentzen, A., and E, W.}
\newblock Solving high-dimensional partial differential equations using deep
  learning.
\newblock {\em Proc. Natl. Acad. Sci. USA 115}, 34 (2018), 8505--8510.

\bibitem{he2015delving}
{\sc He, K., Zhang, X., Ren, S., and Sun, J.}
\newblock Delving deep into rectifiers: {S}urpassing human-level performance on
  imagenet classification.
\newblock In {\em Proceedings of the IEEE international conference on computer
  vision\/} (2015), pp.~1026--1034.

\bibitem{he2016deep}
{\sc He, K., Zhang, X., Ren, S., and Sun, J.}
\newblock Deep residual learning for image recognition.
\newblock In {\em Proceedings of the IEEE conference on computer vision and
  pattern recognition\/} (2016), pp.~770--778.

\bibitem{he2016identity}
{\sc He, K., Zhang, X., Ren, S., and Sun, J.}
\newblock Identity mappings in deep residual networks.
\newblock In {\em European conference on computer vision\/} (2016), Springer,
  pp.~630--645.

\bibitem{hornung2020spacetime}
{\sc Hornung, F., Jentzen, A., and Salimova, D.}
\newblock Space-time deep neural network approximations for high-dimensional
  partial differential equations.
\newblock {\em arXiv:2006.02199\/} (2020), 52 pages.

\bibitem{jentzen2018proof}
{\sc Jentzen, A., Salimova, D., and Welti, T.}
\newblock A proof that deep artificial neural networks overcome the curse of
  dimensionality in the numerical approximation of {K}olmogorov partial
  differential equations with constant diffusion and nonlinear drift
  coefficients.
\newblock {\em Commun. Math. Sci. 19}, 5 (2021), 1167--1205.

\bibitem{kolmogoroff1931analytischen}
{\sc Kolmogoroff, A.}
\newblock \"{U}ber die analytischen {M}ethoden in der
  {W}ahrscheinlichkeitsrechnung.
\newblock {\em Math. Ann. 104}, 1 (1931), 415--458.

\bibitem{krylov2006stochastic}
{\sc Krylov, N.~V., R\"{o}ckner, M., and Zabczyk, J.}
\newblock {\em Stochastic {PDE}'s and {K}olmogorov equations in infinite
  dimensions}, vol.~1715 of {\em Lecture Notes in Mathematics}.
\newblock Springer-Verlag, Berlin; Centro Internazionale Matematico Estivo
  (C.I.M.E.), Florence, 1999.

\bibitem{muller2020space}
{\sc M{\"u}ller, J.}
\newblock On the space-time expressivity of {R}es{N}ets.
\newblock In {\em ICLR 2020 Workshop on Integration of Deep Neural Models and
  Differential Equations\/} (2020).

\bibitem{orhan2018skip}
{\sc Orhan, E., and Pitkow, X.}
\newblock {Skip Connections Eliminate Singularities}.
\newblock In {\em International Conference on Learning Representations\/}
  (2018).

\bibitem{panigrahi2015navigation}
{\sc Panigrahi, P.~K., Ghosh, S., and Parhi, D.~R.}
\newblock Navigation of autonomous mobile robot using different activation
  functions of wavelet neural network.
\newblock {\em Arch. Control Sci. 25(61)}, 1 (2015), 21--34.

\bibitem{petersen2018optimal}
{\sc Petersen, P., and Voigtlaender, F.}
\newblock Optimal approximation of piecewise smooth functions using deep
  {R}e{LU} neural networks.
\newblock {\em Neural Networks 108\/} (2018), 296--330.

\bibitem{qin2019data}
{\sc Qin, T., Wu, K., and Xiu, D.}
\newblock Data driven governing equations approximation using deep neural
  networks.
\newblock {\em J. Comput. Phys. 395\/} (2019), 620--635.

\bibitem{ramachandran2017searching}
{\sc Ramachandran, P., Zoph, B., and Le, Q.~V.}
\newblock Searching for {A}ctivation {F}unctions.
\newblock {\em arXiv:1710.05941\/} (2017), 13 pages.

\bibitem{wang2020influence}
{\sc Wang, Y., Li, Y., Song, Y., and Rong, X.}
\newblock The influence of the activation function in a convolution neural
  network model of facial expression recognition.
\newblock {\em Appl.\ Sci. 10}, 5 (2020), 1897.

\bibitem{wu2016wider}
{\sc Wu, Z., Shen, C., and {van den Hengel}, A.}
\newblock {Wider or Deeper: Revisiting the ResNet Model for Visual
  Recognition}.
\newblock {\em Pattern Recognition 90\/} (2019), 119--133.

\bibitem{zhang2017residual}
{\sc Zhang, K., Sun, M., Han, T.~X., Yuan, X., Guo, L., and Liu, T.}
\newblock Residual networks of residual networks: {M}ultilevel residual
  networks.
\newblock {\em IEEE Transactions on Circuits and Systems for Video Technology
  28}, 6 (2017), 1303--1314.

\end{thebibliography}

\pagebreak

\appendix
\phantomsection
\addcontentsline{toc}{section}{Appendix}
\addtocontents{toc}{\protect\setcounter{tocdepth}{0}}

\end{document}